\documentclass[]{interact}

\usepackage{epstopdf}
\usepackage[caption=false]{subfig}

\usepackage[square, numbers]{natbib}
\renewcommand\bibfont{\fontsize{10}{12}\selectfont}

\theoremstyle{plain}
\newtheorem{theorem}{Theorem}[section]
\newtheorem{lemma}[theorem]{Lemma}

\theoremstyle{definition}

\theoremstyle{remark}

\usepackage{latexsym,array,delarray,epsfig, color,amsfonts,youngtab,yfonts,inslrmin, textcomp, txfonts, pxfonts}
\usepackage{graphicx, epsfig, textcomp, txfonts, pxfonts}
\usepackage{amsmath,amsthm,amsfonts,amssymb,amsxtra, mathtools} 
\usepackage{comment, cancel}
\usepackage{mathrsfs}
\usepackage{mathdots}
\usepackage{multirow}
\usepackage{tikz-cd}

\theoremstyle{plain}
\newtheorem{thm}{Theorem}
\newtheorem{prop}[thm]{Proposition}
\newtheorem{cor}[thm]{Corollary}

\newtheorem*{thm*}{Theorem}
\newtheorem*{lemma*}{Lemma}
\newtheorem*{prop*}{Proposition}
\newtheorem*{cor*}{Corollary}
\newtheorem*{conj*}{Conjecture}

\theoremstyle{definition}
\newtheorem{defn}[thm]{Definition}

\theoremstyle{remark}
\newtheorem*{rmk}{Remark}

\usepackage{amsmath,amsthm,amsfonts,amssymb,amsxtra} 
\usepackage{comment, cancel}
\usepackage{tikz}    

    { \begin{equation*} }
    { \end{equation*} }
\newcommand{\currenumi}{}
	{%
	\renewcommand{\currenumi}{\theenumi}
	\renewcommand{\theenumi}{(\arabic{enumi})}
	
	\begin{enumerate}%
	}
	{%
	\end{enumerate}
	\renewcommand{\theenumi}{\currenumi}%
	}
	{%
	\renewcommand{\currenumi}{\theenumi}
	\renewcommand{\theenumi}{(\roman{enumi})}
	
	\begin{enumerate}%
	}
	{%
	\end{enumerate}
	\renewcommand{\theenumi}{\currenumi}%
	}
	
\newcommand{\C}{\mathbb{C}}

\newcommand{\Z}{\mathbb{Z}}

\newcommand{\cc}{\mathbb{C}}

\newcommand{\zz}{\mathbb{Z}}

 \newcommand\boq{\mathbf q} \newcommand\bor{\mathbf r}

\newcommand{\on}{\operatorname}
\renewcommand{\frak}{\mathfrak}

\newcommand{\g}{\mathfrak{g}}

\newcommand{\h}{\mathfrak{h}}
\newcommand{\sln}{\mathfrak{sl}_{n+1}}

\newcommand{\asln}{\mathfrak{\tilde{sl}}_{n+1}}

\newcommand{\usln}{U(\mathfrak{sl}_{n+1})}

\newcommand{\cta}{\sln[s^{\pm 1}, t^{\pm 1}]}

\newcommand{\Sl}{S_{\ell}}
\newcommand{\asl}{\text{Aff}^{1}(S_{\ell})}
\newcommand{\vl}{V^{\otimes \ell}}

\begin{document}

\title{Schur-Weyl Duality for Toroidal Algebras of Type $A$}

\author{\name{Vyjayanthi Chari\textsuperscript{a}\thanks{The project was made possible by a SQuaRE at the American Institute of Mathematics. The authors thank AIM for providing a supportive and mathematically rich environment. VC also acknowledges support from the Infosys Visiting Chair position at IISc.}, \ Lauren Grimley\textsuperscript{b},\  Zongzhu Lin\textsuperscript{c}, \ Chad R. Mangum\textsuperscript{d},  \\Christine Uhl\textsuperscript{e}, \ Evan Wilson\textsuperscript{f}}
\affil{\textsuperscript{a}University of California, Riverside; \textsuperscript{b}The University of Oklahoma; \textsuperscript{c}Kansas State University; \textsuperscript{d}Clemson University; \textsuperscript{e}St. Bonaventure University; \textsuperscript{f}Brightpoint Community College}}

\maketitle
\begin{abstract}
We state and prove an analog of the Schur-Weyl duality for a quotient of the classical $2$-toroidal Lie algebra of type $A$. We then provide a method to extend this duality to the $m$-toroidal case, $m > 2$. 
\end{abstract}

\begin{keywords}
    Schur-Weyl duality, Toroidal Lie algebras, Multiloop Lie algebras

\noindent{\textbf{2020 MATHEMATICS SUBJECT CLASSIFICATION}} \\
    17B67, 17B65, 20G43

\noindent \textit{Note: This article has been accepted for publication in Communications in Algebra, published by Taylor \& Francis.}
\end{keywords}

\section{Introduction}
Schur-Weyl duality is a classical result (see for instance   \cite{Fulton}) in representation theory which connects  the representation theory  of the symmetric group $\Sl$ with that of the general linear Lie algebra $\mathfrak {gl}_{n+1}$ (or special linear Lie algebra $\mathfrak {sl}_{n+1}$). It has numerous extensions including, among others, dualities between finite dimensional representations of $\mathfrak sp_{2n}, \mathfrak {so}_{n}$ and the Brauer algebra \cite{Dipper}, \cite{Daugherty} (with further extensions in this vein in \cite{Calvert}, \cite {Calverts}, \cite {Doty}); dualities between the loop algebra $\mathfrak {sl}_{n+1}\otimes \mathbb C[t,t^{-1}]$ and the algebra $\mathbb C [\mathbb Z^{\ell} \rtimes \Sl]$ \cite {Flicker}; dualities between quantum/superalgebras and various Hecke or Iwahori-Hecke algebras, as appropriate \cite{Chari}, \cite{Ehrig}, \cite{Flickers}, \cite{Hu}, \cite {Moon}, \cite{VV96}; and related dualities over other rings and fields \cite{Benson-Doty}, \cite{Cruz}, \cite{deConcini-Procesi}, \cite{Krause}. The  duality is  expressed either as a double-centralizer theorem or as an equivalence of categories for certain values of $\ell$ and $n$. Their study is important not only for their theoretical interest but also because they arise in certain constructions of irreducible representations (see for instance \cite{Goodman} for some classical results).

Finite dimensional representations of loop algebras have been studied extensively. These representations are interesting because they are in general not completely reducible. 
There are many interesting families of indecomposable 
 finite dimensional representations of loop algebras such as  Kirillov-Reshetikhin modules which originally arose in the study of integrable models. Other well--known examples are the local Weyl modules, which were originally studied in \cite{ChariPressley}. These are certain universal finite--dimensional modules and can be realized as Demazure modules in level one representations of the affine Lie algebra.
 
There are two main results in this paper. The first establishes a Schur-Weyl duality between finite dimensional representations of the two-loop algebra $\mathfrak {sl}_{n+1} \otimes \mathbb C[t^{\pm 1},s^{\pm 1}]$ (a quotient of the 2-toroidal Lie algebra used in \cite{Jing}, \cite{Moody}, \cite{VV98}) and finite dimensional representations of a classical (i.e. non-quantum) version of the toroidal Hecke algebra of \cite{VV96}. The second main result allows us to extend the first main result to the $m$-toroidal case for $m > 2$. In future work we hope to understand the analog of these constructions in twisted and quantum cases.

The paper is structured as follows. We first recall in \S \ref{classicaluntwistedaffine} the known Schur-Weyl dualities for the $\sln$ and level 0 $\asln$ cases. In \S \ref{classicaluntwistedtoroidal} we generalize these results to $\mathfrak {sl}_{n+1} \otimes \mathbb C[t^{\pm 1},s^{\pm 1}]$, giving our first main result. Finally in \S \ref{SWmultiloop} we use a different approach to give our second main result, a generalization to the case of more than two loops.

\section{Schur-Weyl Duality for $L(\sln)$}\label{classicaluntwistedaffine}

\subsection{Preliminary Definitions} In this section, we recall the Schur-Weyl duality result in the case of classical affine Lie algebras of type $A$, which is \cite{Flicker} Theorem 1.1. 

Recall that $\asln$ is the Lie algebra generated by $x_i^{\pm}, h_i$, where $i \in [0 , n ]$ (for integers $m,p$ with $m \leq p$, the notation $[m,p]$ means $\{ m, m+1, \ldots , p \}$), subject to relations, 
\begin{gather*} [h_i, h_j] = 0,\ \ [h_i, x_j^{\pm}] = \pm \alpha_j (h_i) x_j^{\pm} = \pm a_{ij} x_j^{\pm},\ \ 
 [x_i^{+}, x_j^{-}] = \delta_{ij} h_i,\ \
(\text{ad }x_i^{\pm})^{1-a_{ij}} x_j^{\pm} = 0,\end{gather*}
where $i,j \in [ 0 , n ]$, and  $a_{ij}$ are the entries of the Cartan matrix of type $A_n^{(1)}$.
The Lie algebra $\sln$ is the subalgebra generated by $x_i^\pm, h_i$, with $i \in [1,n]$. Let $L(\sln) :=\sln \otimes \C[t,t^{-1}]$, where $t$ is an indeterminate, be the Lie algebra  with the commutator given by $[a\otimes f, b\otimes g]=[a,b]\otimes fg$. Then $\asln$ is the universal central extension of $L(\sln)$ with one--dimensional center spanned by $h_0+\cdots +h_n$.  

Let $\mathfrak h$ be the subalgebra spanned by the elements $h_i$, $1\le i\le n$ and define $\varepsilon_i\in\mathfrak  h^*$ by requiring   $\varepsilon_i(h_j)=\delta_{i,j} - \delta_{i,j+1}$, and set $\lambda_i = \varepsilon_1 + \cdots + \varepsilon_i$. Then  $R^+=\{\varepsilon_i-\varepsilon_j: 1\le i<j\le n+1\}$ is a set of positive roots for $\mathfrak{sl}_{n+1}$ with respect to $\mathfrak h$ and $\alpha_i=\varepsilon_i-\varepsilon_{i+1}$ for $1\le i\le n$ is a set of simple roots for the pair $\h, \sln$.

\begin{defn}\label{classicallimitahlq}
Define $\asl$ to be the unital associative algebra over $\cc$ with generators $\sigma^{\pm 1}_i$, \ $i \in [1, \ell-1 ]$, and $\mathbf{y}_p$,  \ $p \in [ 1 , \ell ]$ with relations
\begin{gather*}\sigma_k \sigma^{-1}_k = \sigma^{-1}_k \sigma_k = 1,\ \ 
\sigma_k \sigma_{k+1} \sigma_k = \sigma_{k+1} \sigma_k \sigma_{k+1},\ \ 
\sigma_k \sigma_j = \sigma_j \sigma_k ,\ \ \vert k-j \vert > 1, \ \ 
 \sigma_k^2 = 1, \\ 
 \mathbf{y}_m \mathbf{y}_m^{-1} = \mathbf{y}_m^{-1} \mathbf{y}_m = 1,\ \ 
 \mathbf{y}_m \mathbf{y}_p = \mathbf{y}_p \mathbf{y}_m,\\ 
\mathbf{y}_m \sigma_k = \sigma_k \mathbf{y}_m,\ \  m \notin \{ k, k+1 \},\ \
\sigma_k \mathbf{y}_k \sigma_k = \mathbf{y}_{k+1},
\end{gather*}
for all applicable $m,p, k,j$.
\end{defn}

\noindent An elementary argument shows that the subalgebra generated by $\sigma_i$, $1\le i\le \ell-1$ is isomorphic to the group algebra of the symmetric group on $\ell$ letters and $\asl$ is  the group algebra of $\Z^{\ell} \rtimes \Sl$.

\subsection{Classical and Affine Schur-Weyl Duality}\label{classicalaffineswd}
Throughout this paper $V$ will denote the representation of $\asln$ with basis $\{ v_1 , \ldots , v_{n+1} \}$ and with action
\begin{gather*} x_i^{+}.v_r = \delta_{r,i+1} v_{r-1},\ \  x_i^{-}.v_r = \delta_{r,i} v_{r+1},\ \ h_i.v_r = (\delta_{r,i} - \delta_{r,i+1}) v_{r},\\ 
x_{0}^{-}.v_r = \delta_{r,n+1} v_{1},\ \ x_{0}^{+}.v_r = \delta_{r,1} v_{n+1},\ \ 
 h_{0}.v_r = -(h_1 + \cdots + h_n).v_{r} = (\delta_{r,n+1} - \delta_{r,1}) v_{r},\end{gather*}
where  $i \in [ 1 , n ]$ and $v_{0} = v_{n+2} = 0$. Note that $V$ is just the natural representation of $\mathfrak{sl}_{n+1}$.  Given $\ell\ge 1$ the representation $ V^{\otimes \ell}$ is completely reducible as an $\mathfrak{sl}_{n+1}$--module. Any $\mathfrak{sl}_{n+1}$--module isomorphic to a submodule of this $\ell$--fold tensor product will be said to be of degree $\ell$. \\\\
It is well--known that the natural permutation action of $\Sl$ on $V^{\otimes\ell}$ commutes with the action of $\mathfrak{sl}_{n+1}$.
\\\\
We recall the classical Schur-Weyl duality for $\sln$. See for instance \cite{Fulton} \S 6 and \S 15.

\begin{prop}\label{slnleftaction}
Equip $M \otimes_{\C[\Sl]} \vl$ with the natural left $\usln$--module structure induced by that on $\vl$ (that is, $x.(m \otimes \mathbf{w}) = m \otimes x.\mathbf{w}$ for $m \in M, \mathbf{w} \in \vl$ and $x\in\mathfrak{sl}_{n+1}$). Then, if $\ell \leq n$, the functor $M \to M \otimes_{\C[\Sl]} \vl$ is an equivalence from the category of right finite dimensional $\Sl$--modules to the category of left finite dimensional $\usln$--modules of degree $\ell$.\hfill\qedsymbol
\end{prop}

Next  we state the Schur-Weyl duality for $\asln$ at level 0 from \cite{Flicker} Theorem 1.1. 
\begin{thm}\label{nonquantumuntwistedequiv}
Fix integers $\ell \geq 0, n \geq 1$. There exists a functor $\mathcal{F}$ from the category $\text{Rep } \asl$ of finite dimensional right $\asl$--modules to the category $\text{Rep}_{\ell} (\asln)$ of finite dimensional level 0 left $\asln$--modules of degree $\ell$ defined as follows. If $M$ is an $\asl$--module, then the $\sln$--module structure on $\mathcal{F}(M) = M \otimes_{\C[\Sl]} \vl$ extends to an $\asln$--module structure by
$$\displaystyle x_{0}^{\pm}.(m \otimes\mathbf{v}) = \sum_{j=1}^{\ell} m.\mathbf{y}_j^{\pm 1} \otimes (x_{\theta}^{\mp})_j.\mathbf{v},\ \ 
  h_{0}.(m \otimes\mathbf{v}) =  m \otimes (-h_{\theta})^{\otimes \ell}.\mathbf{v},$$
where $m \in M$ and  $\mathbf{v} \in \vl$. 
Further, $\mathcal{F}$ is an equivalence of categories if $\ell \leq n$.\hfill\qedsymbol
\end{thm}

\begin{rmk}
\cite{Flicker} states that $\mathcal{F}$ is not an equivalence when $\ell = n$; however, our $n$ is $n-1$ in \cite{Flicker}. We have adjusted the statement of the theorem accordingly. Note also the alternative statement \cite{Flicker} Theorem 1.2 which makes explicit the action of an arbitrary element of $L(\sln)$.
\end{rmk}

\section{Schur-Weyl Duality for $\sln[s^{\pm 1}, t^{\pm 1}]$}\label{classicaluntwistedtoroidal}

In this section, we state and prove a related Schur-Weyl duality statement for a quotient in the classical 2-toroidal case by extending Theorem \ref{nonquantumuntwistedequiv}. We begin by giving the requisite definitions and preliminary results following a similar structure to \S \ref{classicaluntwistedaffine}.

From now on, assume $\ell \leq n$. Let $s$ and $t$ be indeterminates and  let $\sln[s^{\pm 1}, t^{\pm 1}] := \sln \otimes \cc[s^{\pm 1}, t^{\pm 1}]$ with the obvious Lie bracket. In what follows, we shall set
$$x^\pm _i(k)=x^{\pm}_i\otimes s^k,\ \ i\in [1, n ],\ \ x_0^\pm(k)= x_{\varepsilon_1-\varepsilon_{n+1}}^\pm \otimes s^kt^{\pm 1},\ \ h_j(k)=h_j\otimes s^k,\ \ j\in [0, n], $$
where we remind the reader that $h_0=-(h_1+\cdots+h_n).$ It is known from \cite{Moody} Proposition 3.5 that these elements generate $\sln[s^{\pm 1}, t^{\pm 1}]$ with defining relations, 
\begin{gather}\label{tor1} [h_i(k) , h_j (m)] = 0,\ \ 
 [h_i (k) , x_j^{\pm} (m)] = \pm a_{ij} x_j^{\pm} (k+m),\\ \label{tor2}
 [x_i^{+} (k) , x_j^{-} (m)] = \delta_{ij}  h_i (k+m)  ,\\ \label{tor3}
 [x_i^{+} (k)  , x_i^{+} (m) ] = 0 = [x_i^{-} (k) , x_i^{-} (m) ],\ \ 
(\text{ad } x_i^{\pm} (0) )^{1-a_{ij}} x_j^{\pm} (m)  = 0,\ \ i\ne j,\end{gather}
 where $k,m \in \zz$, and the $a_{ij}$ are the entries of the Cartan matrix of type $A_n^{(1)}$.
\\\\
Extend the action of   $\sln$ on   $V$ to an action of $\sln[s^{\pm 1}, t^{\pm 1}]$ by $$x\otimes f(s,t)(v)= f(1,1)xv,\ \ x\in\sln,\ \ f\in\mathbb C[s^{\pm 1}, t^{\pm 1}].$$
In particular, we have \begin{gather*} h_j (k).v = (h_j \otimes s^k).v =  h_j. v,
 \ \ x_0^{\pm} (k).v = (x_{\theta}^{\mp} \otimes s^k t^{\pm 1}).v =  x_{\theta}^{\mp} .v,\ \
 x_i^{\pm} (k).v = (x_i^{\pm} \otimes s^k).v =  x_i^{\pm} .v,\end{gather*}
for $i \in [1, n ]$, $j \in [0, n ]$. These representations are studied in more detail in \cite{Lau}.

\begin{defn}\label{classicaluntwistedtoroidalhecke}
Define $\on{Aff}^2(\Sl)$ to be the unital associative algebra over $\C$ with generators $\sigma^{\pm 1}_i$, $i \in [1, \ell-1 ]$, $\mathbf{x}_r$, $\mathbf{y}_r$,  $r \in [1, \ell ]$ and relations
\begin{gather*}
 \sigma_k \sigma^{-1}_k = \sigma^{-1}_k \sigma_k = 1,\ \ 
\sigma_k \sigma_{k+1} \sigma_k = \sigma_{k+1} \sigma_k \sigma_{k+1},\\
 \sigma_k \sigma_j = \sigma_j \sigma_k,\ \  \vert k-j \vert > 1,\ \
 \sigma_k^2 = 1,\ \\ 
 \mathbf{y}_m \mathbf{y}_m^{-1} = \mathbf{y}_m^{-1} \mathbf{y}_m = 1,\ \ \mathbf{x}_m \mathbf{x}_m^{-1} = \mathbf{x}_m^{-1} \mathbf{x}_m = 1,\ \ 
\mathbf{y}_m \mathbf{y}_p = \mathbf{y}_p \mathbf{y}_m,\ \ \mathbf{x}_m \mathbf{x}_p = \mathbf{x}_p \mathbf{x}_m,\\
 \mathbf{y}_m \sigma_k = \sigma_k \mathbf{y}_m,\ \ \mathbf{x}_m \sigma_k = \sigma_k \mathbf{x}_m,\ \ m \notin \{ k, k+1 \}\\ \sigma_k \mathbf{y}_k \sigma_k = \mathbf{y}_{k+1},\ \ \sigma_k^{-1} \mathbf{x}_k \sigma_k^{-1} = \mathbf{x}_{k+1},\ 
 \\ \mathbf{y}_1 \cdots \mathbf{y}_{\ell} \mathbf{x}_1 = \mathbf{x}_1 \mathbf{y}_1 \cdots \mathbf{y}_{\ell},\ \  \ \ \mathbf{x}_1 \mathbf{y}_2= \mathbf{y}_2 \mathbf{x}_1
\end{gather*}
for $k \in [1, \ell-2 ]$ in the second relation; $k,j \in [1, \ell-1 ]$ otherwise; and $m,p \in [1, \ell ]$.
\end{defn}
\begin{rmk}
The subalgebra generated by $\sigma_i, \mathbf{y}_j$ is isomorphic to $\asl \cong \C [\Z^{\ell} \rtimes \Sl]$ from Definition \ref{classicallimitahlq}, and the same is true for the subalgebra generated by $\sigma_i, \mathbf{x}_j$. Further, if $\ell=1$ in Definition \ref{classicaluntwistedtoroidalhecke}, most relations above would be trivial, and the algebra $\cc[\mathbf{x}_1^{\pm 1}, \mathbf{y}_1^{\pm 1}]$ would be the coordinate algebra of the algebraic torus $\left( \cc^* \right)^2$.
\end{rmk}

\begin{rmk}Definition \ref{classicaluntwistedtoroidalhecke} specializes that of the toroidal Hecke algebra given in \cite{VV96} Definition 1.1. From \cite{VV96} Remark 1.4, the toroidal Hecke algebra takes the affine Hecke algebra (Def. 3.1 in \cite{Chari}) and adds a second set of polynomial generators. In this article, starting from \cite{VV96} Definition 1.1, we take the classical limit $q \to 1$ and take $\C$ in place of $\cc[s^{\pm 1}, t^{\pm 1}]$ (that is, we set $s=1=t$).
\end{rmk}

\begin{prop}\label{torheckecommute}
We have $\mathbf{x}_m\mathbf{y}_p= \mathbf{y}_p \mathbf{x}_m$ for all $m,p\in[1,\ell]$.
\end{prop}

\begin{proof} If $\ell=1$, then the equality is just a defining relation.  If $\ell=2$ then we use the defining relations as follows: $$\mathbf{y}_2 \mathbf{y}_1 \mathbf{x}_1= \mathbf{y}_1 \mathbf{y}_2 \mathbf{x}_1= \mathbf{x}_1 \mathbf{y}_1 \mathbf{y}_2= \mathbf{x}_1 \mathbf{y}_2 \mathbf{y}_1 =\mathbf{y}_2 \mathbf{x}_1 \mathbf{y}_1 \implies \mathbf{y}_1 \mathbf{x}_1= \mathbf{x}_1 \mathbf{y}_1,$$ and further use of the defining relations again give the proposition for all $m,p$.
If $\ell >2$ we proceed by a double induction. Observe $$\sigma_2 \mathbf{x}_1= \mathbf{x}_1\sigma_2,\ \ \sigma_2 \mathbf{x}_1 \mathbf{y}_2\sigma_2=\sigma_2 \mathbf{y}_2 \mathbf{x}_1\sigma_2\implies \mathbf{x}_1 \mathbf{y}_3= \mathbf{y}_3 \mathbf{x}_1,$$ and an induction proves the result for all $p\ge 2$. It follows that  $\mathbf{x}_1 \mathbf{y}_1\cdots \mathbf{y}_\ell= \mathbf{y}_1 \mathbf{x}_1 \mathbf{y}_2\cdots \mathbf{y}_\ell$ which implies that $\mathbf{x}_1\mathbf{y}_1= \mathbf{x}_1\mathbf{y}_1$.
\\\\
Assuming now that the result holds for all $1\le j<m$ we prove it for $m$. Here we find that  $$\sigma_{m-1}^{-1}\mathbf{x}_{m-1} \mathbf{y}_p\sigma_{m-1}^{-1}=\sigma_{m-1}^{-1} \mathbf{y}_p\mathbf{x}_{m-1}\sigma_{m-1}^{-1} \implies \mathbf{x}_m\mathbf{y}_p= \mathbf{y}_p\mathbf{x}_m,\ \ p\ne m,m+1,$$ and hence it suffices to prove that $\mathbf{x}_m\mathbf{y}_{m+1}= \mathbf{y}_{m+1}\mathbf{x}_m$ and $\mathbf{x}_m\mathbf{y}_m= \mathbf{y}_m\mathbf{x}_m$. To prove the first we use the inductive hypothesis on $m$ and then the defining relations to get
$$\sigma_{m-1}^{-1}\mathbf{x}_{m-1}\mathbf{y}_{m+1} \sigma_{m-1}^{-1}= \sigma_{m-1}^{-1}\mathbf{y}_{m+1}\mathbf{x}_{m-1}\sigma_{m-1}^{-1}\implies \mathbf{x}_m\mathbf{y}_{m+1}= \mathbf{y}_{m+1}\mathbf{x}_m.$$
The second equality follows by using the inductive hypothesis to see that $\sigma_{m-1}\mathbf{y}_{m-1}\mathbf{x}_{m-1} \sigma_{m-1}^{-1}=\sigma_{m-1}\mathbf{x}_{m-1}\mathbf{y}_{m-1}\sigma_{m-1}^{-1}$ and then using the  defining relations.
\end{proof}
Let $S_\ell$ act diagonally  on $\Z^{\ell} \times \Z^{\ell}$. The proof of the following is now elementary, with the identification $\sigma_c \mapsto (\mathbf{0}, \mathbf{0}, \sigma_c), y_q \mapsto (y_q , \mathbf{0}, 1), x_r \mapsto (\mathbf{0} , x_r, 1)$, and $1 \mapsto (\mathbf{0} , \mathbf{0}, 1)$, where $\mathbf{0} = (0, \ldots ,0) \in \Z^{\ell}$.

\begin{lemma}\label{torheckeisom}
$\on{Aff}^2(\Sl) \cong \C [(\Z^{\ell} \times \Z^{\ell}) \rtimes \Sl] = \C [\Z^{\ell} \rtimes (\Z^{\ell} \rtimes \Sl)]$.\hfill\qedsymbol
\end{lemma}

\subsection{The main result} Let $\text{Rep}\left(\on{Aff}^2(\Sl)\right)$ be the category of finite dimensional right $\on{Aff}^2(\Sl)$--modules, and let $\text{Rep}_{\ell}(\sln[s^{\pm 1}, t^{\pm 1}])$ be the category of finite dimensional modules for $\sln[s^{\pm 1}, t^{\pm 1}]$ of degree $\ell$, namely those which occur in the $\ell$--th tensor power of $V$. Given $a\in \sln[s^{\pm 1}, t^{\pm 1}]$ and an element $\mathbf{v}= v_1\otimes\cdots\otimes  v_\ell\in V^{\otimes \ell}$ we denote $$(a)_j\mathbf{v}= v_1\otimes\cdots v_{j-1}\otimes av_j\otimes v_{j+1}\otimes\cdots\otimes v_\ell.$$
\begin{prop} Suppose that $M$ is a finite dimensional right module for $\on{Aff}^2(\Sl)$ and set $$\mathcal F(M)= M\otimes_{\mathbb C[\Sl]} V^{\otimes\ell}.$$ The following formulae define an action of $\sln[s^{\pm 1}, t^{\pm 1}]$ on $\mathcal
F(M)$: for $i\in[0,n]$ and $k\in\mathbb Z$,
$$h_i (k)(m \otimes \mathbf{v}) = \sum_{1 \leq j \leq \ell} m\mathbf{x}_j^k \otimes\big(h_i (k) \big)_j\mathbf{v},\ \  \ \ x_i^{\pm} (k)(m \otimes \mathbf{v}) = \sum_{1 \leq j \leq \ell} m\mathbf{x}_j^k\mathbf{y}_j^{\pm\delta_{i,0}} \otimes \big(x_i^{\pm} (k)\big)_j\mathbf{v},$$
for all $m\in M$ and $\mathbf{v} \in V^{\otimes \ell}$.
Moreover, if $M'$ is another finite dimensional module for $\on{Aff}^2(\Sl)$ and $f: M\to M'$ is a map of $\on{Aff}^2(\Sl)$--modules, the assignment $m\otimes \mathbf{v}\to f(m)\otimes \mathbf{v}$ extends to  a homomorphism of $\sln[s^{\pm 1}, t^{\pm 1}]$--modules.
\end{prop}
\begin{proof} 
A direct verification using the relations in $\on{Aff}^2(\Sl)$ shows that $\mathcal{F}$ is well-defined on objects (that is, $x.(m.\sigma_i \otimes_{\C[\Sl]} \mathbf{v}) = x.(m \otimes_{\C[\Sl]} \sigma_i.\mathbf{v})$ for each $i$, with $x$ a generator of $\sln[s^{\pm 1}, t^{\pm 1}]$, and $m \in M, \mathbf{v} \in V^{\otimes \ell}$) as in \cite{Flicker} \S5. It is elementary to see that $\mathcal{F}$ is well-defined on morphisms, a homomorphism of $\sln[s^{\pm 1}, t^{\pm 1}]$--modules, and  that $\mathcal{F}(\textbf{1}_{\text{Rep}\left(\on{Aff}^2(\Sl)\right)}) = \textbf{1}_{\text{Rep}_{\ell}\left(\sln[s^{\pm 1}, t^{\pm 1}]\right)}$ and $\mathcal{F}(f \circ g) = \mathcal{F}(f) \circ \mathcal{F} (g)$. Thus $\mathcal{F}$ exists and is a well-defined, covariant functor.
To check that the formulae give an action of $\sln[s^{\pm 1}, t^{\pm 1}]$ it suffices to prove that the relations in equations \eqref{tor1}-\eqref{tor3} are satisfied. But this is a straightforward, if tedious, computation. The final statement of the proposition is also obvious.

\end{proof}

We now state our main theorem of this section. 

\begin{thm}\label{classicaluntwistedtoroidalequiv} Retain the notation established so far. 
There exists a covariant functor $\mathcal{F}: \text{Rep}\left(\on{Aff}^2(\Sl)\right) \to \text{Rep}_{\ell}(\sln[s^{\pm 1}, t^{\pm 1}])$ which sends $M\to\mathcal F(M)$  and the morphism $f:M\to M'$ to the morphism $\mathcal{F}(f)$. Further, $\mathcal{F}$ is an equivalence of categories when $\ell \leq n$.
\end{thm}

\begin{rmk}
Recall, by \cite{Moody} Proposition 3.5, 
$$x^\pm _i(k)=x^{\pm}_i\otimes s^k,\ \ i\in [1, n ],\ \ x_0^\pm(k)= x_{\varepsilon_1-\varepsilon_{n+1}}^\pm \otimes s^kt^{\pm 1},\ \ h_j(k)=h_j\otimes s^k,\ \ j\in [0, n] $$
generate $\sln[s^{\pm 1}, t^{\pm 1}]$.  Taking $k=0$ reduces Theorem \ref{classicaluntwistedtoroidalequiv} to Theorem \ref{nonquantumuntwistedequiv}. Theorem \ref{classicaluntwistedtoroidalequiv} therefore extends the statement for $\asln$ to $\cta$. Note also that $\mathbf{x}$ (resp. $\mathbf{y}$) corresponds to the variable $s$ (resp. $t$).
\end{rmk}

The proof of Theorem \ref{classicaluntwistedtoroidalequiv} will be accomplished through a sequence of Lemmas in \S \ref{functordefinesrepn}, \S \ref{everyrepnfromfunctor}.
The $\ell=0$ case follows from \cite{Flicker} Cor. 15.4 where $(s,t)$ acts as $ (1,1)$, so we assume $\ell \geq 1$ below. Note that, when $\ell=1$, then $\Sl$ and many of the sums below are trivial.

\begin{subsection}{$\mathcal{F}$ is a Fully Faithful Functor}\label{functordefinesrepn}

\begin{lemma}
The functor $\mathcal{F}: \text{Rep}\left(\on{Aff}^2(\Sl)\right) \to \text{Rep}_{\ell}\left(\sln[s^{\pm 1}, t^{\pm 1}]\right)$ is fully faithful.
\end{lemma}

\begin{proof}
The classical Schur-Weyl duality gives us that the functor
$$\mathcal{F}^0: \C[\Sl] \textrm{-mod} \rightarrow \sln \textrm{-mod},\ \   M \mapsto M \otimes_{\C[\Sl]} V^{\otimes \ell},\ \ f\to f\otimes 1,$$ 
is  fully faithful. Since any morphism $\phi: M\to N$ of $\on{Aff}^2(\Sl)$--modules is also a morphism of $\Sl$--modules and $\mathcal F(\phi)=\mathcal F^0(\phi)$ the injectivity of $\mathcal F$ is immediate.\\\\
For surjectivity, let $\tilde\phi: M \otimes_{\C[\Sl]} V^{\otimes \ell} \rightarrow N \otimes_{\C[\Sl]} V^{\otimes \ell}$  be a map of $\mathfrak{sl}_{n+1}[t^{\pm 1}, s^{\pm 1}]$--modules and hence a map of $\mathfrak{sl}_{n+1}$--modules. Using the classical Schur--Weyl duality  we can choose an $\Sl$--module map  $\phi:M\to N$
such that $\mathcal F^0(\phi)=\tilde\phi$. It suffices to check that $\phi$ is a map of $\on{Aff}^2(\Sl)$--modules and for this, we must check that $\phi$ commutes with the generators $\mathbf{x}_j$, $\mathbf{y}_j$ for  $j\in[1,\ell]$.

  By considering $h_i(k)$ commuting with $\tilde\phi$ (using the classical Schur-Weyl duality), then $\tilde\phi$ commutes with $\mathbf{x}_p^k$
action on $m \otimes v_1 \otimes \cdots \otimes v_\ell$.

An argument identical to the one given in \cite{Flicker}, (using the action of  $x_0^+(0)$ acting on $m \otimes_{\C[\Sl]} v_1 \otimes \cdots \otimes v_\ell$ for suitable choices of $v_j$, $1\le j\le \ell$), we get that $\phi$ commutes with $\mathbf{y}_1$.

To see that $\phi$ commutes with  $\mathbf{y}_r$ for $r>1$ assume that we have proved the result for $r-1$ and use
\begin{gather*}
 \phi(m\mathbf{y}_{r+1})= \phi(m\sigma_r\mathbf{y}_r\sigma_r)= \phi(m)\sigma_r\mathbf{y}_r\sigma_r  =\phi(m)\mathbf{y}_{r+1},
\end{gather*}
where we have used the inductive hypothesis for the second equality.
\end{proof}

\end{subsection}

\begin{subsection}{$\mathcal{F}$ is an Equivalence of Categories}\label{everyrepnfromfunctor}
To complete the proof of Theorem \ref{classicaluntwistedtoroidalequiv}, it remains only to show that $\mathcal{F}$ is an equivalence of categories when $\ell \leq n$. We must show that every finite dimensional left $\mathfrak{sl}_{n+1}[t^{\pm 1}, s^{\pm 1}]$--module of degree $\ell$ is in the image of $\mathcal F$.

Thus let $W$ be a $\mathfrak{sl}_{n+1}[t^{\pm 1}, s^{\pm 1}]$--module occurring in $V^{\otimes \ell}$. Regarding this as a module for the subalgebra $\mathfrak{sl}_{n+1}[t^{\pm 1}]$ we use \cite{Flicker} to  choose an $\asl$--module $M$ so that we have an isomorphism of $\mathfrak{sl}_{n+1}[t^{\pm 1}]$--modules $W\cong M\otimes_{\C[\Sl]} V^{\otimes \ell}.$ It suffices to prove that we can define an action of $\on{Aff}^2(\Sl)$ on $M$ so that the preceding isomorphism is one of 
$\mathfrak{sl}_{n+1}[t^{\pm 1},s^{\pm 1}]$--modules.

The following will be useful. Let $\{ v_1, \ldots , v_{n+1} \}$ denote the standard basis of $V$ and for $1\le i\ne j\le n+1$, let $x_{i,j}\in \mathfrak{sl}_{n+1}$ be defined by requiring $x_{i,j} v_s=\delta_{j,s} v_i$.

\begin{lemma}\label{injectivelemma}
 Suppose that $i_1,\ldots, i_\ell\in[1,n+1]$ are distinct and set $\mathbf{v} = v_{i_1} \otimes \cdots \otimes v_{i_{\ell}} \in V^{\otimes \ell}$. Then 
 $V^{\otimes \ell} = U(\sln).\mathbf{v}$ and the map $ M\to M\otimes_{\cc[S_\ell]}V^{\otimes\ell} $ defined by $m \mapsto m \otimes\mathbf{v}$ is injective for any $\Sl$--module $M$. 
\end{lemma}

\begin{proof} For $j_\ell\ne i_\ell$ 
we have $$x_{j_\ell, i_\ell}\mathbf{v}= v_{i_1}\otimes \cdots \otimes v_{i_{\ell-1}}\otimes v_{j_\ell}.$$ It follows that $$v_{i_1}\otimes \cdots\otimes v_{i_{\ell}-1}\otimes V\subset U(\mathfrak{sl}_{n+1})\mathbf{v}.$$
Repeating with $x_{j_{r}, i_{r}}$ for $1\le r\le \ell-1$ ,   the  first assertion of the lemma follows. The injectivity is now immediate since  $m\otimes \mathbf{v}=0$ would imply that $m\otimes V^{\otimes \ell}=0$. 
\end{proof}

Retain the notation established so far. Let $M$ be a right finite dimensional $\asl$--module such that $\text{Res}^{\sln[s^{\pm 1}, t^{\pm 1}]}_{\asln} W$ is of the form $M \otimes_{\C[\Sl]} V^{\otimes \ell}, \ell \leq n$. For $i \in \{ 0 , \ldots , n \}$, set
	$$\hat{a}_{i} = \begin{cases}
		v_1 \otimes \cdots \otimes \hat{v}_i \otimes \hat{v}_{i+1} \otimes \cdots \otimes v_{\ell+1}, \text{ if } 1 \le i\le \ell;\\
		v_1 \otimes \cdots \otimes v_{\ell-1}, \text{ if } \ell+1\le i \le n;\\
		v_2 \otimes \cdots \otimes v_{\ell}, \text{ if } i=0,\\
	\end{cases}$$
	where $\hat{v}_i$ means that $v_i$ is missing in the tensor product (for $\ell=2$, $\hat{a}_{i} = v_3$ if $1 \leq i \leq \ell$). Here, the indices are considered cyclical, i.e., $v_0$ is defined to be $v_{n+1}$. Set
	$$u^-_{i,1} = v_{i+1} \otimes \hat{a}_{i}, \hspace{5mm} u^+_{i,1} = v_{i} \otimes \hat{a}_{i}.$$
	(If $\ell=1$, set $u^-_{i,1} = v_{i+1}, u^+_{i,1} = v_{i}.$) For $j\in \{1,2,\dots,\ell\}$ define 
	$$
	u^{\pm}_{i,j}=\sigma(1,j)u^{\pm}_{i,1}
	$$
	where $\sigma(1,j)$ is the element of $\asl$ corresponding to the transposition $(1,j)$, i.e.,
	$$\sigma=\sigma_1\sigma_{2}\cdots \sigma_{j-2}\sigma_{j-1}\sigma_{j-2}\cdots \sigma_2\sigma_1.$$
 Using the above notations, we can prove the following.
 
\begin{lemma}\label{Fprop72}
	For each $i\in [0, n] , k \in \Z$, there exist $\alpha_{j,x_i^\pm(k)},\alpha_{j,h_i(k)} \in \text{End}_{\cc}(M)$ with
	\begin{align*}
		x_i^\pm(k).\left(m \otimes_{\C[\Sl]} u^{\mp}_{i,j} \right) &= \alpha_{j,x_i^\pm(k)}(m) \otimes_{\C[\Sl]} u^{\pm}_{i,j}\ \\
		h_i(k).\left(m \otimes_{\C[\Sl]} u^{\pm}_{i,j} \right) &= \pm\alpha_{j,h_i(k)}(m) \otimes_{\C[\Sl]} u^{\pm}_{i,j}\ 
	\end{align*}
	and moreover,
	\begin{align*}
		\big(x_i^\pm(k)\big)_j.u^{\mp}_{i,j} &= u^{\pm}_{i,j};\\
		\big(h_i(k)\big)_j .u^{\pm}_{i,j} &= \pm u^{\pm}_{i,j}.
	\end{align*}
\end{lemma}

\begin{proof}
	For a fixed $i$, the set $\{ \tau.u^{\pm}_{i,1} \mid \tau \in \Sl \}$ spans the $\sln$-weight space $V^{\otimes \ell}_{\lambda_{u^{\pm}_{i,1}}}$ of weight $\lambda_{u^{\pm}_{i,1}}$ of $V^{\otimes \ell}$ where
	\begin{align*}
		\lambda_{u^{+}_{i,1}}&=\begin{cases}
			\sum_{p=1}^{\ell+1} \varepsilon_p-\varepsilon_{i+1},\text{ if } 1\le i\le \ell;\\
			\sum_{p=1}^{\ell-1} \varepsilon_p+\varepsilon_{i},\text{ if } \ell+1 \le i\le n;\\
			\sum_{p=2}^{\ell} \varepsilon_p+\varepsilon_{n+1},\text{ if } i=0
		\end{cases}
	\end{align*}
	and $\lambda_{u^{-}_{i,1}}=\lambda_{u^{+}_{i,1}}-\alpha_i$. Indeed, using the action from \S \ref{classicalaffineswd}, for any $h_i (0), i = 1, \ldots , n$, we have
	$$h_i (0) . v_j = 
	(\delta_{j,i} - \delta_{j,i+1}) v_j = \varepsilon_j (h_i) v_j.$$
	Thus $v_j$ is of weight $\varepsilon_j$ and hence $h_j (k).\tau.u^{\pm}_{i,1} = \lambda_{u^{\pm}_{i,1}} (h_j) \tau.u^{\pm}_{i,1}$ and so $\tau. u^{\pm}_{i,1}$ is of $\sln$-weight $\lambda_{u^{\pm}_{i,1}}$.
	
	Now, a direct computation shows that $h_i (0).\tau.u^{+}_{i,1} = \lambda_{u^{\pm}_{i,1}}(h_i (0))\tau.u^{+}_{i,1}=\tau.u^{+}_{i,1}$. Similarly, we see that $h_i (0).\tau.u^{-}_{i,1}= -\tau.u^{-}_{i,1}$. It follows from these calculations and Theorem \ref{nonquantumuntwistedequiv} that for every $m \in M$, we have $h_i (0).\left(m \otimes_{\C[\Sl]} u^{\pm}_{i,1} \right) = \pm m \otimes_{\C[\Sl]} u^{\pm}_{i,1}$, so $m \otimes_{\C[\Sl]} u^{\pm}_{i,1}$ is in the weight space $W_{\lambda_{u^{\pm}_{i,1}}}$.
	
	The relations in $\cta$ show that $h_i (k)$ must take $m\otimes_{\C[\Sl]} u^{\pm}_{i,1}$ to a vector of weight $\lambda_{u^{\pm}_{i,1}}$ since if $w$ has weight $\lambda_{u^{\pm}_{i,1}}$ we must have, for $1\le j \le n$,  $h_j(0)h_i(k).w=h_i(k)h_j(0).w=\lambda_{u^{\pm}_{i,1}}(h_j)h_i(k).w$. And since $\{ \tau.u^{\pm}_{i,1} \mid \tau \in \Sl \}$ spans the subspace of $V^{\otimes \ell}$ of its associated weight, then $h_i (k)$ must take $m \otimes_{\C[\Sl]} u^{\pm}_{i,1}$ to a sum of terms of the form $m_\tau\otimes \tau.u^{\pm}_{i,1}$ for some $m_\tau$ (well-defined since $W$ is an $\sln[s^{\pm 1}, t^{\pm 1}]$--module). Now, for each $m \in M$ we have:
	\begin{align*}
		h_i (k).\left(m \otimes_{\C[\Sl]} u^{\pm}_{i,1} \right) &= \sum_{\tau \in \Sl} m_{\tau} \otimes_{\C[\Sl]} \lambda_{u^{\pm}_{i,1}}(h_i) \tau.u^{\pm}_{i,1} \\
		& = \sum_{\tau \in \Sl} m_{\tau}.\tau \otimes_{\C[\Sl]} \lambda_{u^{\pm}_{i,1}}(h_i) u^{\pm}_{i,1}\\
		& = \pm m' \otimes_{\C[\Sl]} u^{\pm}_{i,1}
	\end{align*}
	where $m' = \sum_{\tau \in \Sl} m_{\tau}.\tau$. Note that this $u^{\pm}_{i,1}$ meets the conditions of Lemma \ref{injectivelemma}. Therefore each such $m'$ can be recovered from $m' \otimes_{\C[\Sl]} u^{\pm}_{i,1} $. Then there exists $\alpha_{1,h_i(k)} \in \text{End}_{\cc}(M)$ with the property that, for each $m \in M$, we can find $m'$ such that $m' = \alpha_{1,h_i(k)} (m)$.
	Defining $\alpha_{j,h_i(k)}(m)=\alpha_{1,h_i(k)}(m.\sigma).\sigma^{-1}$ we see that 
	\begin{align*}
		h_i (k). \left(m \otimes_{\C[\Sl]} u^{\pm}_{i,j} \right)&=h_i (k). \left(m. \sigma \otimes_{\C[\Sl]} \sigma^{-1}. u^{\pm}_{i,j} \right)\\
		&=h_i (k). \left(m.\sigma \otimes_{\C[\Sl]} u^{\pm}_{i,1} \right)\\
		&=\pm(\alpha_{1,h_i(k)}(m.\sigma) \otimes_{\C[\Sl]} u^{\pm}_{i,1})\\
		&=\pm(\alpha_{1,h_i(k)}(m.\sigma).\sigma^{-1} \otimes_{\C[\Sl]} \sigma.u^{\pm}_{i,1})\\
		&=\pm(\alpha_{j,h_i(k)}(m) \otimes_{\C[\Sl]} u^{\pm}_{i,j}).
	\end{align*}
	
	Now assume $w$ has weight $\lambda$. By relations in $\cta$, for any $w \in W, 1\le i \le n$, we have 
	\begin{align*}
		h_i (0) x_j^{\pm} (k') . w &= (x_j^{\pm} (k') h_i (0) \pm \alpha_j (h_i(0)) x_j^{\pm} (k')).w\\
		& = (\lambda \pm \alpha_j) (h_i(0)) x_j^{\pm} (k').w.
	\end{align*}
	Thus $x_j^{\pm} (k')$ adds $\pm \alpha_j = \pm (\varepsilon_j - \varepsilon_{j+1})$ to the weight if $j = 1, \ldots , n$, or adds $\alpha_0 = \mp (\varepsilon_1 - \varepsilon_{n+1})$ to the weight if $j = 0$. In particular, by direct calculation, the weight of $x_j^{\pm} (k').u^{\mp}_{j,1}$ is the same as the weight of $u^{\pm}_{j,1}$. We conclude that for every $m \in M$, $x_j^{\pm} (k').\left(m \otimes_{\C[\Sl]} u^{\mp}_{j,1} \right) = \sum_{\tau \in \Sl} m_{\tau} \otimes_{\C[\Sl]} \tau.u^{\pm}_{j,1}$ for some $m_{\tau} \in M$. An argument similar to the one above proves the existence of the $\alpha_{j,x_i^\pm(k)}$ with the properties in the Lemma statement.

\end{proof}

\begin{lemma}\label{Fprop73}
For all $z \in \{ h_i (k), x_i^{\pm} (k) \mid i = 0, \ldots , n \text{ and } k \in \Z \}$, $m \in M$ (with $M$ as in Lemma \ref{Fprop72}), and $\mathbf{v} \in V^{\otimes \ell}$, we have 
$$z.(m \otimes_{\C[\Sl]} \mathbf{v}) = \sum_{1 \leq p \leq \ell} \alpha_{p,z}(m) \otimes_{\C[\Sl]} z_p. \mathbf{v}.$$
\end{lemma}

\begin{proof}
Let $z = x_i^{+} (k), i = 2, \ldots , n$ (in particular, $n \geq 2$; note that, if $n=1$, then $\ell=1$, in which case Lemma \ref{Fprop73} reduces to Lemma \ref{Fprop72}). Other cases are similar and are omitted.

Let $r$ (resp. $s$) be the number of tensor factors in $\mathbf{v}$ containing $v_i$ (resp. $v_{i+1}$). Given $\mathbf{v}$, write $\bar{r}(\mathbf{v})$ for the number $r$ in $\mathbf{v}$, and similarly for $\bar{s}$, giving $s$. We will prove the statement of the lemma for all $\mathbf{v} \in V^{\otimes \ell}$ via induction. First induct on $r$.

Base case: if $r=s=0$, then the sum $\sum_{1 \leq p \leq \ell} \alpha_{p,x_i^{+} (k)}(m) \otimes_{\C[\Sl]} \big(x_i^{+} (k) \big)_p . \mathbf{v} = 0$ since $x_i^{+} (k).\mathbf{v}=0$. Now, $x_i^{\pm} (k).(m \otimes_{\C[\Sl]} \mathbf{v}) = \frac{1}{2} \left( h_i(k) x_i^{\pm}(0) - x_i^{\pm}(0) h_i(k) \right).(m \otimes_{\C[\Sl]} \mathbf{v})$. By Theorem \ref{nonquantumuntwistedequiv}, $x_i^{\pm}(0) .(m \otimes_{\C[\Sl]} \mathbf{v}) =  m \otimes_{\C[\Sl]} x_i^{\pm}(0).\mathbf{v} = 0$, and $h_i(k)$ preserves the weight space of $m \otimes_{\C[\Sl]} \mathbf{v}$, so $x_i^{\pm}(0)$ applied to $h_i (k) (m \otimes_{\C[\Sl]} \mathbf{v})$ is 0. Hence both terms are 0, proving the Lemma when $r=s=0$.

Inductive step: assume that $\mathbf{v}$ is such that $\bar{r}(\mathbf{v})=r+1$ and $\bar{s}(\mathbf{v})=s$. Then there exists $k_0 \in \{ 1 , \ldots , n \}$ such that:
\begin{enumerate}
    \item $v_{k_0}$ does not appear in $\mathbf{v}$ (since $\ell \leq n$);
    \item we may assume that $k_0 = i-1$ ($n \geq 2$). Otherwise there exists $\rho \in W(\sln)$, the Weyl group of $\sln$, such that $\rho(k_0) = i-1$ and $\rho(i)=i, \rho(i+1)=i+1$. Then $\bar{r}(\rho(\mathbf{v}))=r+1, \bar{s}(\rho(\mathbf{v}))=s$ and $\mathbf{v}$ has no copies of $v_{i-1}$.
\end{enumerate}
If $\ell \geq 2$, take $p$ such that $i_p=i$ (which is possible if $\ell \geq 2$, since $r+1>0$) and define $\mathbf{w}$ such that $w_{i_q}=v_{i_q}$ for $q \neq p$ and $w_{i_p}=v_{i-1}$. Then $\mathbf{v} = x^{-}_{i-1} (0) \mathbf{w}$ and $\bar{r}(\mathbf{w})=r, \bar{s}(\mathbf{w})=s$, so we can apply the induction hypothesis to $\mathbf{w}$. So
\begin{align*}
    x_i^{+} (k). (m \otimes_{\C[\Sl]} \mathbf{v}) &= x_i^{+} (k) x_{i-1}^{-} (0). (m \otimes_{\C[\Sl]} \mathbf{w}) = x_{i-1}^{-} (0) x_i^{+} (k) .(m \otimes_{\C[\Sl]} \mathbf{w}) \\
    &= x_{i-1}^{-} (0) \sum_{1 \leq p \leq \ell} \alpha_{p, x_i^{+} (k)}(m) \otimes_{\C[\Sl]} \big(x_i^{+} (k) \big)_p . \mathbf{w} \\
    &= \sum_{1 \leq p \leq \ell} \alpha_{p, x_i^{+} (k)}(m) \otimes_{\C[\Sl]} \big(x_i^{+} (k) \big)_p . \mathbf{v},
\end{align*}
which proves the statement of the lemma for all $r$ by induction. (If $\ell=1$, take $\mathbf{v}=v_i$ and $\mathbf{w}=v_{i-1}$ and perform the same computation above.) A similar argument establishes the statement for all $s$, completing the proof of the lemma.

\end{proof}

\begin{lemma}\label{alphaequality}
	For each $i,j \in [1,n], p \in [1,\ell]$, and $k \in \Z$, we have 
	$$
		\alpha_{p,z_i(k)}=\alpha_{p,z'_{j}(k)}=\left(\alpha_{p,z_i(1)}\right)^k
	$$ 
	where $z,z'$ is $x^{\pm}$ or $h$.
\end{lemma}
\begin{proof}
	We first show that $\alpha_{p,h_i(k)}=\alpha_{p,h_j(k)}$. This is trivial if $n=1$, so assume $n \geq 2$. For our first case, assume that $\ell \geq 2$ and without loss of generality that $1\le i<j\le n$. Let 
	$$\mathbf v=\sigma(1,p)(v_{i}\otimes v_{i+2}\otimes a_i)$$ 
	where $a_i\in V^{\otimes (\ell-2)}$ is a pure tensor having indices all distinct and not among $\{i,i+1,i+2\}$.
	
	We have:
	\begin{equation}
	h_i (k).(m \otimes_{\C[\Sl]} \mathbf{v})=\alpha_{p,h_i(k)}(m)\otimes_{\C[\Sl]} \mathbf{v}.\label{eq:cartanaction}
	\end{equation}
	Now act on the right side of (\ref{eq:cartanaction}) by $x_{i+1}^-(0)x_{i+1}^+(0)$ to get:
	\begin{align*}
	x_{i+1}^-(0)x_{i+1}^+(0).(\alpha_{p,h_i(k)}(m)\otimes_{\C[\Sl]} \mathbf{v})=\alpha_{p,h_i(k)}(m) \otimes_{\C[\Sl]} \mathbf{v}.
	\end{align*}
	Acting on the left side of (\ref{eq:cartanaction}) by $x_{i+1}^-(0)x_{i+1}^+(0)$ gives
	\begin{align*}
		x_{i+1}^-(0)x_{i+1}^+(0)h_i (k).(m\otimes_{\C[\Sl]} \mathbf{v})&=x_{i+1}^-(0)(x_{i+1}^+(k) + h_i (k)x_{i+1}^+(0)).(m\otimes_{\C[\Sl]} \mathbf{v})\\
		&=x_{i+1}^-(0)x_{i+1}^+(k).(m\otimes_{\C[\Sl]} \mathbf{v})\\
		&=h_{i+1}(k).(m\otimes_{\C[\Sl]} \mathbf{v})
	\end{align*}
	and hence $\alpha_{p,h_i(k)}(m)=\alpha_{p,h_{i+1}(k)}(m)$ by Lemma \ref{injectivelemma}. Now repeat the same process to get $\alpha_{p,h_i(k)}(m)=\alpha_{p,h_j(k)}(m)$ for $1 \le i<j \le n$.
	
	If instead $\ell=1$, assume $1\le j<i \le n$ and take $\mathbf{v} = v_i$. Then compute (\ref{eq:cartanaction}) and act by $x_{i-1}^-(0)x_{i-1}^+(0)$ to show by a similar process that $\alpha_{p,h_i(k)}(m)=\alpha_{p,h_{i-1}(k)}(m)$.
	
	Now, we show that $\alpha_{p,h_i(k)}=\alpha_{p,x^{\pm}_i(k)}, 1\le i \le n$. Let
	$$\mathbf w=\sigma(1,p)(v_{i}\otimes b_i)$$ 
	where $b_i\in V^{\otimes (\ell-1)}$ is a pure tensor having indices all distinct and not among $\{i,i+1\}$. (If $\ell=1$, take $\mathbf{w} = v_i$.)
	We compute
	\begin{align*}
		h_i(k).(m\otimes_{\C[\Sl]} \mathbf{w})
		&=(x_i^+(k)x_i^-(0)-x_i^-(0)x_i^+(k)).(m\otimes_{\C[\Sl]} \mathbf{w})\\
		&=\alpha_{p,x_i^+(k)}(m)\otimes_{\C[\Sl]} \mathbf w.
	\end{align*}
	Again, Lemma \ref{injectivelemma} shows that $\alpha_{p,h_i(k)}=\alpha_{p,x^+_i(k)}$. A similar argument gives $\alpha_{p,h_i(k)}=\alpha_{p,x^-_i(k)}$.
	
	We now show by induction the second equality of the lemma statement for the case $n \geq 2, k > 0$ (the other cases are similar). The base case $k=1$ is trivial, so assume the statement for $k \geq 1$. From now on, let $\alpha_{p,z_i(k)}=\alpha_{p,k}$ for $1 \le i \le n$. We have
     $$
     [x_{2}^+(k),h_1(1)]. (m \otimes_{\C[\Sl]} u_{2,p}^-)=x_{2}^+(k+1). (m \otimes_{\C[\Sl]} u_{2,p}^-)=\alpha_{p,k+1}(m) \otimes_{\C[\Sl]} u_{2,p}^+ .
     $$
     The left hand side above is equal to $\alpha_{p,k}\alpha_{p,1}(m) \otimes_{\C[\Sl]} u_{2,p}^+ = \left(\alpha_{p,1} \right)^k \alpha_{p,1}(m) \otimes_{\C[\Sl]} u_{2,p}^+$ by the inductive hypothesis, and so the statement follows. By similar arguments, we see that  $\alpha_{p,z_{0}(k)}=\alpha_{p,1}^k\alpha_{p,z_{0}(0)}$, as well as $\alpha_{p,-k}=\alpha_{p,-1}^k$ and $\alpha_{p,z_{0}(-k)}=\alpha_{p,-1}^{k}\alpha_{p,z_{0}(0)}$, completing the proof.
    
\end{proof}

\begin{lemma}\label{Fprop74}
For each $z \in \{ h_i (k), x_j^{\pm} (k) \mid i \in [0,n]; j \in [1, n]; k \in \Z \}$, $m \in M$ (with $M$ as in \S \ref{everyrepnfromfunctor}), set $m.\mathbf{x}_p^k = \alpha_{p,z}(m)$. For $x_0^{\pm} (k)$, set $m.\mathbf{x}_p^k \mathbf{y}_p^{\pm 1} = \alpha_{p,x_0^{\pm} (k)}(m)$. Here, $\alpha_{p,z}$ is as in Lemma \ref{Fprop73}. This assignment defines a right $\on{Aff}^2(\Sl) \cong \C [(\Z^{\ell} \times \Z^{\ell}) \rtimes \Sl]$--module structure on $M$, extending its $\asl \cong \C [\Z^{\ell} \rtimes \Sl]$ structure.
\end{lemma}

\begin{proof}

The $\asl$--module structure on $M$ is inherited from Theorem \ref{nonquantumuntwistedequiv} (the $k=0$ case), so we only need to show that the relations from Definition \ref{classicaluntwistedtoroidalhecke} involving $\mathbf{x}_j^{\pm 1}$ in $\on{Aff}^2(\Sl)$ hold. By Lemma \ref{injectivelemma}, it is sufficient to show this on $m \otimes_{\C[\Sl]} \mathbf{v}$ for arbitrary $m$ but a particular choice of $\mathbf{v}$ such that $\mathbf{v}$ has $\ell$ distinct components.

All relations except $\mathbf{y}_1 \cdots \mathbf{y}_{\ell} \mathbf{x}_1 = \mathbf{x}_1 \mathbf{y}_1 \cdots \mathbf{y}_{\ell}$ and $\mathbf{x}_1 \mathbf{y}_2 = \mathbf{y}_2 \mathbf{x}_1$ follow from Lemma \ref{alphaequality}, noting that all $h_i (k)$ are commuting. We show that the remaining two relations hold.

First we show $m.\mathbf{y}_2 \mathbf{x}_1 = m.\mathbf{x}_1 \mathbf{y}_2$. Note that this relation does not appear if $\ell=1$, so we will assume $\ell \geq 2$ (and hence $n \geq 2$; recall $\ell \leq n$).

Take $\mathbf{v} = v_2 \otimes v_1 \otimes v_3 \otimes v_4 \otimes \cdots \otimes v_{\ell}$ and $\mathbf{v}' = v_2 \otimes v_{n+1} \otimes v_3 \otimes v_4 \otimes \cdots \otimes v_{\ell}$. Now
$$(h_1(1) x_0^{+} (0)). (m \otimes_{\C[\Sl]} \mathbf{v}) = h_1 (1).(m.\mathbf{y}_2 \otimes_{\C[\Sl]} \mathbf{v}') = -m.\mathbf{y}_2 \mathbf{x}_1 \otimes_{\C[\Sl]} \mathbf{v}'.$$
By relations in $\mathfrak{sl}_{n+1}[s^{\pm 1}, t^{\pm 1}]$ we can write the left hand side as
$$(x_0^{+} (0) h_1(1) - x_0^{+} (1)). (m \otimes_{\C[\Sl]} \mathbf{v}) = x_0^{+}(0).\left((m.\mathbf{x}_2 - m.\mathbf{x}_1) \otimes_{\C[\Sl]} \mathbf{v} \right) - m.\mathbf{x}_2 \mathbf{y}_2 \otimes_{\C[\Sl]} \mathbf{v}'$$
$$ =(m.\mathbf{x}_2 \mathbf{y}_2 - m.\mathbf{x}_1 \mathbf{y}_2) \otimes_{\C[\Sl]} \mathbf{v}' - m.\mathbf{x}_2 \mathbf{y}_2 \otimes_{\C[\Sl]} \mathbf{v}' = -m.\mathbf{x}_1 \mathbf{y}_2 \otimes_{\C[\Sl]} \mathbf{v}'.$$
Thus Lemma \ref{injectivelemma} implies that $m.\mathbf{y}_2 \mathbf{x}_1 = m.\mathbf{x}_1 \mathbf{y}_2$ for all $m \in M$. 

Now we show that $m.(\mathbf{y}_1 \cdots \mathbf{y}_{\ell} \mathbf{x}_1) = m.(\mathbf{x}_1 \mathbf{y}_1 \cdots \mathbf{y}_{\ell})$. When combined with the previous relation and induction, it is sufficient to demonstrate $m.\mathbf{y}_1 \mathbf{x}_1 = m.\mathbf{x}_1 \mathbf{y}_1$. 

Take $\mathbf{v} = v_1 \otimes \cdots \otimes v_{\ell}$ and $\mathbf{v}' = v_{n+1} \otimes v_2 \otimes v_3 \otimes \cdots \otimes v_{\ell}$.

Then
$$h_0 (1) x_0^+ (0) .(m \otimes_{\C[\Sl]} \mathbf{v}) = m.(\mathbf{y}_1 \mathbf{x}_1) \otimes_{\C[\Sl]} \mathbf{v}'.$$
By the relations in $\mathfrak{sl}_{n+1}[s^{\pm 1}, t^{\pm 1}]$, this is equal to 
$$(x_0^+ (0) h_0 (1) + 2 x_0^+ (1)). (m \otimes_{\C[\Sl]} \mathbf{v}) = x_0^+ (0). (-m.\mathbf{x}_1 \otimes_{\C[\Sl]} \mathbf{v}) + 2 m.\mathbf{x}_1 \mathbf{y}_1 \otimes_{\C[\Sl]} \mathbf{v}' = m.(\mathbf{x}_1 \mathbf{y}_1) \otimes_{\C[\Sl]} \mathbf{v}'.$$
Therefore, Lemma \ref{injectivelemma} implies that $m.(\mathbf{x}_1 \mathbf{y}_1) = m.(\mathbf{y}_1 \mathbf{x}_1)$ for all $m \in M$.

\end{proof}

This completes the proof of Theorem \ref{classicaluntwistedtoroidalequiv}.
\end{subsection}

\section{Schur-Weyl Duality for Multiloop Lie Algebras}\label{SWmultiloop}

In this section we introduce a technique to prove a  Schur-Weyl duality for $ L^m(\g)=\g \otimes \C[t_1^{\pm 1}, \cdots, t_{m}^{\pm 1}]$ with $m\ge 2$ where as usual $\mathfrak g$ is of type $A_n$. 

\subsection{Gluing multiloop Lie algebras from loop Lie algebras} For any complex  Lie algebra $\mathfrak a $, set 
 $ L^m(\mathfrak a)=\mathfrak a  \ \otimes \C[t_1^{\pm 1}, \cdots, t_{m}^{\pm 1}]$. Clearly  $ L^{m+1}(\mathfrak a )=L(L^m(\mathfrak a))$ and for $i\in[1,m]$ we let  $  L_i(\g)=\g \otimes \cc[t_i^{\pm 1}]$.
 Let $$ \iota_i: L_i(\g)\to L^m(\g),\ \ \ \iota_{i, 0}: \g\to L_i(\g),\ \ \ \iota^m_0: \g\to L^m (\g)$$ be the obvious embedddings of Lie algebras and note that  $ \iota^m_0=\iota_i\circ \iota_{i,0}$ for all $i\in[1,m]$.

\begin{prop}\label{liftingprop}
Let $ \frak g$ be any perfect Lie algebra (i.e. $\mathfrak g = [\mathfrak g,\mathfrak g]$).  Let $L$ be any Lie algebra over $\cc$.  Assume that we have  system of Lie algebra  homomorphisms $ \rho_i: L_i(\frak g)\to L$ such that, 
\begin{itemize}
\item[(i)] $\rho_0:=\rho_i\circ \iota_{i,0}=\rho_j\circ \iota_{j,0}$,
\item[(ii)] for all $c \in [1,m]$ and all subsets $\{i_1,\dots, i_c\}$ of $\{1,\dots, m\}$ we have that
$$\sum_{x_1,\dots, x_c\in \mathfrak{g}} [x_1, [x_2, \cdots[x_{c-1}, x_c]\cdots ]]=0 \implies
$$
$$
\sum_{x_1,\dots, x_c \in \mathfrak g} [\rho_{i_1}(x_1\otimes t_{i_1}^{p_1}), [\rho_{i_2}(x_2\otimes t_{i_2}^{p_2}), \cdots[\rho_{i_{c-1}}(x_{c-1}\otimes t_{i_{c-1}}^{p_{c-1}}), \rho_{i_c}(x_c\otimes t_{i_c}^{p_c})]\cdots ]]=0,\ \ \ (p_1,\cdots ,p_c)\in\mathbb Z^c,
$$
and
\item[(iii)]
\begin{align*}
&[\rho_{i_1}(x_1\otimes t_{i_1}^{p_{i_1}}), [\rho_{i_2}(x_2\otimes t_{i_2}^{p_{i_2}}), \cdots[\rho_{i_{c-1}}(x_{c-1}\otimes t_{i_{c-1}}^{p_{i_{c-1}}}), \rho_{i_c}(x_c\otimes t_{i_c}^{p_{i_c}})]\cdots ]]\nonumber\\
&=[\rho_{j_1}(x_1\otimes t_{j_1}^{p_{j_1}}), [\rho_{j_2}(x_2\otimes t_{j_2}^{p_{j_2}}), \cdots[\rho_{j_{c-1}}(x_{c-1}\otimes t_{j_{c-1}}^{p_{j_{c-1}}}), \rho_{j_c}(x_c\otimes t_{j_c}^{p_{j_c}})]\cdots ]]
\end{align*}
whenever $(j_1,j_2,\dots, j_c)$ is a permutation of $(i_1,i_2,\dots, i_c).$
\end{itemize}
Then there is a unique Lie algebra homomorphism $ \rho:L^m(\frak g)\to L$ such that $ \rho_i=\rho\circ \iota_i$.
\end{prop}

\begin{proof} 
In this proof, by $\rho_i(x)$ for $x \in \mathfrak g$ we will mean $(\rho_i \circ \iota_{i,0})(x)$. By (i), this means that $\rho_i(x)=\rho_j(x)$.
We now prove the statement of the proposition by induction on $m$. The statement is vacuously true if $m=1$. For the inductive step, we proceeds as follows.

For all $x\in \frak g$ choose a decomposition $x = \sum_{j} [x_{j}', x_{j}'']$ (a finite sum). We define inductively:
$$
\rho(x \otimes t_1^{p_1}t_2^{p_2} \cdots t_m^{p_m})=\sum_{j} [\rho(x_{j}' \otimes t_1^{p_1} \cdots t_{m-1}^{p_{m-1}}), \rho_m(x_{j}'' \otimes t_m^{p_m})], \hspace{3mm} m>1.$$

We need to show two facts: (1) $\rho$ is well-defined, i.e. it doesn't depend on the choice of decomposition, (2) $\rho$ does indeed define a Lie algebra homomorphism. The main step of this argument is to show that $\rho$, assumed well-defined for $m-1$, can be expressed in the form:
\begin{equation}
\rho(x\otimes t^{\mathbf p} )=\sum_{x_{1}',\dots, x_{m}' \in \mathfrak g}[\rho_{m}(x_{m}'\otimes t_{m}^{p_{m}}), [\rho_{m-1}(x_{m-1}'\otimes t_{{m-1}}^{p_{m-1}}), \cdots[\rho_{2}(x_{2}'\otimes t_{2}^{p_{2}}),\rho_1(x_{1}' \otimes t_{1}^{p_{1}})]\cdots ]] \label{eq:rhobracket}
\end{equation}
where
$$
x=\sum_{x_{1}',\dots, x_{m}' \in \mathfrak g}[x_{m}', [x_{m-1}', \cdots[x_{2}',x_{1}' ]\cdots ]].
$$
This is clear when $m=0,1$ so assume that (\ref{eq:rhobracket}) holds for $m-1$, and compute:
\begin{align*}
&\sum_{j} [\rho(x_{j}' \otimes t_{1}^{p_1} \cdots t_{{m-1}}^{p_{m-1}}), \rho_m(x_{j}'' \otimes t_m^{p_m})]=\sum_{j} [ \rho_m(-x_{j}'' \otimes t_m^{p_m}),[\rho(x_{j}' \otimes t_{1}^{p_1} \cdots t_{{m-1}}^{p_{m-1}})]]\\
&=\sum_j\sum_{x_{1,j}',\dots, x_{m-1,j}' \in \mathfrak g} [\rho_m(-x_{j}''\otimes t_m^{p_m}),[\rho_{m-1}(x_{m-1,j}'\otimes t_{m-1}^{p_{m-1}}), [ \cdots[\rho_{2}(x_{2,j}'\otimes t_{2}^{p_{2}}),\rho_1(x_{1,j}' \otimes t_{1}^{p_{1}})]\cdots ]]]
\end{align*}
where
\begin{align*}
\sum_j\sum_{x_{1,j}',\dots, x_{m-1,j}' \in \mathfrak g} [-x_{j}'',[x_{m-1,j}', [ \cdots[x_{2,j}',x_{1,j}' ]\cdots ]]]=\sum_j[-x_j'',x_j']=x
\end{align*}
as desired.

Now, we use (\ref{eq:rhobracket}) to prove that $\rho$ is well-defined on $L^m(\frak g)$. In particular, it is independent of the chosen decomposition. Let $x=\sum_i [x_{i,1}',x_{i,1}'']=\sum_i[x_{i,2}',x_{i,2}'']$ be two distinct decompositions and use (\ref{eq:rhobracket}) to write
\begin{align*}
&\sum_{j=1}^2(-1)^j\sum_i [\rho(x_{i,j}'\otimes t_1^{p_1}t_2^{p_2}\cdots t_{m-1}^{p_{m-1}}),\rho_m(x_{i,j}''\otimes t_m^{p_m})]\nonumber\\
&=\sum_{j=1}^2(-1)^j\sum_i\sum_{x_{i,j,1}',\dots, x_{i,j,m-1}' \in \mathfrak g}[-\rho_m(x_{i,j}''\otimes t_m^{p_m}),[\rho_{m-1}(x_{i,j,m-1}'\otimes t_{m-1}^{p_{m-1}}), [\cdots[\rho_{2}(x_{i,j,2}'\otimes t_{2}^{p_{2}}),\rho_1(x_{i,j,1}' \otimes t_{1}^{p_{1}})]\cdots ]]].\label{eq:decomp}
\end{align*}
Now, observe that, by our choices
$$
\sum_{j=1}^2(-1)^j\sum_i\sum_{x_{i,j,1}',\dots, x_{i,1,m-1}' \in \mathfrak g}[-x_{i,j}'',[x_{i,j,m-1}', [\cdots [x_{i,j,2}',x_{i,j,1}']\cdots ]]]=\sum_{j=1}^2(-1)^j\sum_i[x_{i,j}',x_{i,j}'']=x-x=0.
$$
Therefore condition (ii) applies to show:
$$
\sum_i [\rho(x_{i,1}'\otimes t_1^{p_1}t_2^{p_2}\cdots t_{m-1}^{p_{m-1}}),\rho_m(x_{i,1}''\otimes t_m^{p_m})]=\sum_i [\rho(x_{i,2}'\otimes t_1^{p_1}t_2^{p_2}\cdots t_{m-1}^{p_{m-1}}),\rho_m(x_{i,2}''\otimes t_m^{p_m})].
$$
Hence, the decomposition of $x$ does not affect the definition of $\rho.$

Now we show that $\rho$ is a homomorphism of  Lie algebras. 
We want to show that $$
[\rho(x\otimes t^\mathbf s),\rho(y\otimes t^\mathbf r)]=\rho([x,y]\otimes t^{\mathbf s + \mathbf r}),\ \ x,y\in\mathfrak  g, \   \mathbf r, \mathbf s\in\mathbb Z^m.
$$
Write $\mathbf r = (r_1, \ldots , r_m), \mathbf s = (s_1, \ldots , s_m)$. We first check the equality in the case when $s_j=0$ for all $j$. Here the result holds by definition. Then we assume that this has been proved for $s_j=0$ for $j>p$ and prove it for $s_j=0$ with $j>p+1$.
By (\ref{eq:rhobracket}) we write:
$$
\rho(y\otimes t^{\mathbf r} )=\sum_{y_{1}',\dots, y_{m}' \in \mathfrak g}[\rho_m(y_m'\otimes t_m^{r_m}), [\rho_{m-1}(y_{m-1}'\otimes t_{{m-1}}^{r_{m-1}}), \cdots[\rho_2(y_2'\otimes t_2^{r_2}),\rho_1(y_1' \otimes t_1^{r_1})]\cdots ]].
$$
Then:
\begin{align*}
&\sum_{y_{1}',\dots, y_{m}' \in \frak g}[\rho_1(x\otimes t_1^{s_1}),[\rho_m(y_m'\otimes t_m^{r_m}), [\rho_{m-1}(y_{m-1}'\otimes t_{m-1}^{r_{m-1}}), \cdots[\rho_2(y_2'\otimes t_2^{r_2}),\rho_1(y_1' \otimes t_1^{r_1})]\cdots ]]]\\
&=\sum_{y_1',\dots, y_m' \in \mathfrak g}[\rho_m(x\otimes t_m^{r_m}),[\rho_{m-1}(y_{m}'\otimes t_{m-1}^{r_{m-1}}), [ \cdots[\rho_{1}(y_2'\otimes t_1^{r_1}),\rho_1(y_1' \otimes t_1^{s_1})]\cdots ]]]\\
&=\sum_{y_1',\dots, y_m' \in \mathfrak g}[\rho_m(x\otimes t_m^{r_m}),[\rho_{m-1}(y_m'\otimes t_{m-1}^{r_{m-1}}), [ \cdots[\rho_2(y_3'\otimes t_2^{r_2}),\rho_{1}([y_2',y_1']\otimes t_1^{r_1+s_1})]\cdots ]]]\\
&=\rho([x,y]\otimes t^{\mathbf r}t_1^{s_1})
\end{align*}
using (iii) in the second line to permute the $\rho_{i}$ and powers of $t_i$.
Now, assume $\rho$ is a homomorphism for $1\le p < m$. The definition of $\rho$ gives:
$$
\rho(x\otimes t_1^{s_1}\cdots t_{p+1}^{s_{p+1}})=\sum_j [\rho(x_j'\otimes t_1^{s_1}\cdots t_{p}^{s_{p}}),\rho_u(x_j''\otimes t_{p+1}^{s_{p+1}})].
$$
An argument similar to the one given above shows that, for all $z \in \frak g$
$$
[\rho_{p+1}(z\otimes t_{p+1}^{s_{p+1}}), \rho(y\otimes t^{\mathbf r})]=\rho([z,y]\otimes t^{\mathbf r}t_{p+1}^{s_{p+1}}).
$$
Finally, we use the Jacobi Identity and induction to write:
\begin{align*}
    &\left [\sum_j [\rho(x_j'\otimes t_1^{s_1}\cdots t_{p}^{s_{p}}),\rho_{p+1}(x_j''\otimes t_{p+1}^{s_{p+1}})], \rho(y\otimes t^{\mathbf r})\right ]\\
    &=\sum_j([\rho(x_j'\otimes t_1^{s_1}\cdots t_p^{s_p}),[\rho_{p+1}(x_j''\otimes t_{p+1}^{s_{p+1}}), \rho(y\otimes t^{\mathbf r})]]-[\rho_{p+1}(x_j''\otimes t_{p+1}^{s_{p+1}}), [\rho(x_j'\otimes t_1^{s_1}\cdots t_p^{s_p}),\rho(y\otimes t^{\mathbf r})]])\\
    &=\rho\left(\sum_j([x_j',[x_j'',y]]-[x_j'', [x_j',y]])\otimes t^{\mathbf s + \mathbf r}\right )\\
    &=\rho([x,y]\otimes t^{\mathbf r + \mathbf s}).
\end{align*}

Hence the proposition is proved. The maps and algebras involved are collected in the following commutative diagram.


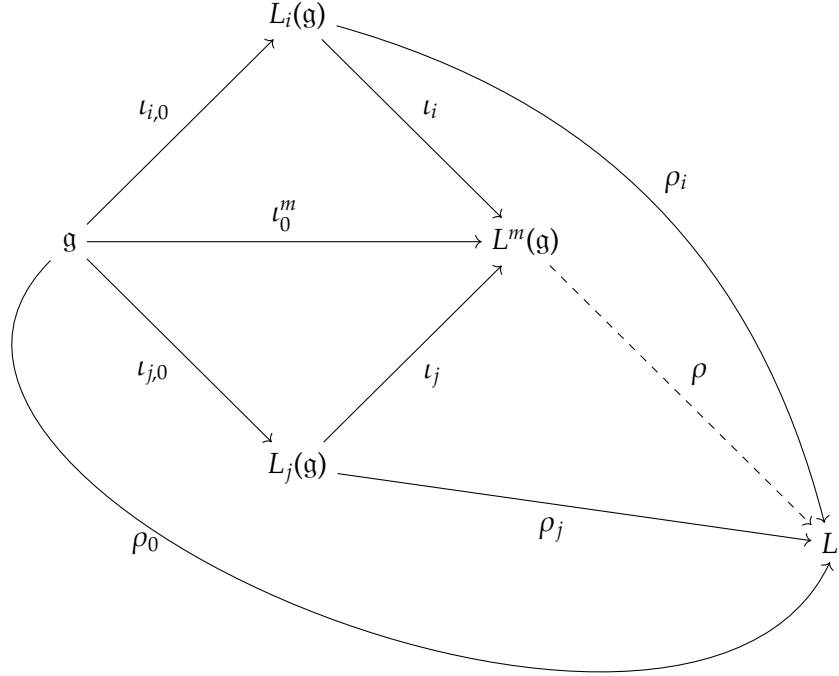
\begin{tikzpicture}[node distance=2cm, auto]
  \node (A) {$L_i(\g)$};
  \node (g) [node distance=3cm, left of=P, below of=A] {$\g$};
  \node (Lm) [node distance=3cm, right of=P, below of=A] {$L^m(\g)$};
   \node (Lj) [node distance=6cm, below of=A] {$L_j(\g)$};
  \node (L) [node distance=4cm, right of=Lm, below of=Lm] {$L$};
  \draw[->] (g) to node {$\iota_0^m$} (Lm);
  \draw[->] (g) to node {$\iota_{i,0}$} (A);
  \draw[->] (A) to node {$\iota_i$} (Lm);
   \draw[->] (g) to node [swap] {$\iota_{j,0}$} (Lj);
   \draw[->] (Lj) to node [swap] {$\iota_j$} (Lm);
   \draw[->, dashed] (Lm) to node {$\rho$} (L);
   \draw[->] (Lj) to node [swap] {$\rho_j$} (L);
    \draw[->, bend left] (A) to node {$\rho_i$} (L);
  \draw[->] (-3.25,-3.25) to[out=225,in=245] (7,-7.25);
  \node at (-2, -7) {$\rho_0$};
\end{tikzpicture}

\end{proof}

We are interested in the case when $\mathfrak g$ is a finite dimensional semisimple Lie algebra. Explicitly, we use the decomposition:
$$
x_i^+ = \frac{1}{2}[h_i,x_i^+], \qquad x_i^- = -\frac{1}{2}[h_i,x_i^-], \qquad h_i = [x_i^+,x_i^-].$$

\begin{rmk}
    In case $\frak{g}$ is a finite dimensional semisimple Lie algebra, it can be checked by a similar argument that condition (iii) is irrelevant. We do not show this fact here.
\end{rmk}

\subsection{Gluing representations for multiloop Lie algebras from loop Lie algebras}
The following is elementary and will be useful for what follows.
\begin{lemma}\label{uniqueliftedmodstruc}
Let $A_i$ for $i=1, 2, \ldots , m$ be a collection of subalgebras of an associative algebra $A$ such that the $A_i$ generate $A$.
\begin{itemize}
\item[(a)] For all $A$--modules $M,N$ with linear map $f : M \to N$, the map $f$ is an $A$-homomorphism if and only if $f$ is an $A_i$-homomorphism for each $i$.
\item[(b)] Given algebra homomorphisms $\phi: A \to B$ and $\phi': A \to B$, if $\phi|_{A_i}= \phi'|_{A_i}$ for each $i$, then $\phi = \phi'$.
\end{itemize}
\end{lemma}

\begin{rmk}
Note that any Lie algebra $L$ can be embedded inside its universal enveloping algebra $U(L)$, and a generating set of $L$ is also a generating set of $U(L)$. In this way Lemma \ref{uniqueliftedmodstruc} also applies to Lie algebras.
\end{rmk}

\begin{cor} Let $W$ be a $\frak g$--module and assume that there exist Lie algebra homomorphisms $\rho_i: L_i(\frak g)\to End W$ 
 satisfying the conditions of Proposition \ref{liftingprop}. 
Then there is a unique $L^m(\frak g)$--module structure on $W$ whose restriction to $L_i$ via $ \iota_i$ is the given $L_i(\frak g)$--module structure. 
\end{cor}

\begin{proof}
The statement is trivial if $m=1$; we will proceed now by induction. Assume the statement for $m=k-1$, with $2 \leq k \in \Z$, and consider $m=k$.

Let $y \otimes t^{\mathbf p_k} \in L^k(\frak g)$ where $t^{\mathbf p_k} := t_1^{p_1} \cdots t_k^{p_k}$ be such that $\displaystyle y=\sum_{y_1,\dots, y_k\in \mathfrak{g}} [y_k, [y_{k-1}, \cdots[y_2, y_1]\cdots ]]$. Let now $\displaystyle y'=\sum_{y_1,\dots, y_{k-1}\in \mathfrak{g}} [y_{k-1}, [y_{k-2}, \cdots[y_2, y_1]\cdots ]]$. Then for $w \in W$, define an action of $y \otimes t^{\mathbf p_k}$ on $w$ by $(y \otimes t^{\mathbf p_k})w := (y_k \otimes t^{p_k})(y' \otimes t^{\mathbf p_{k-1}})w - (y' \otimes t^{\mathbf p_{k-1}}) (y_k \otimes t^{p_k})w$. By the inductive assumption, this action is a well-defined action of $L^k(\frak g)$ on $W$. That its restriction to $L_i$ via $ \iota_i$ is the given $L_i(\frak g)$--module structure follows from Proposition \ref{liftingprop}. Uniqueness follows from Lemma \ref{uniqueliftedmodstruc} (b) since the $L_i(\frak g), i \in [1,k]$, generate $L^k(\frak g)$ as a Lie algebra.

\end{proof}

\subsection{Gluing multi-affine Weyl group from affine Weyl groups} The previous section demonstrated that the multiloop $L^m(\mathfrak{g})$-module structure can be uniquely lifted from a given family of single loop $L_i(\mathfrak{g})$-module structures.  In this section, we show that the same can be said on the $S_{\ell}$ side.  

\begin{defn}\label{classicaluntwistedtoroidalheckem}
Define $\on{Aff}^m(\Sl)$ to be the unital associative algebra over $\C$ with generators $\sigma^{\pm 1}_i$, $i \in [1, \ell-1 ]$, $\mathbf{y}_{1,r}$, ..., $\mathbf{y}_{m,r}$,  $r \in [1, \ell ]$ and relations

\begin{gather*}
 \sigma_k \sigma^{-1}_k = \sigma^{-1}_k \sigma_k = 1,\ \ 
\sigma_k \sigma_{k+1} \sigma_k = \sigma_{k+1} \sigma_k \sigma_{k+1},\\
 \sigma_k \sigma_j = \sigma_j \sigma_k,\ \  \vert k-j \vert > 1,\ \
 \sigma_k^2 = 1,\ \\ 
 \mathbf{y}_{i,s} \mathbf{y}_{i,s}^{-1} = \mathbf{y}_{i,s}^{-1} \mathbf{y}_{i,s} = 1, \ \mathbf{y}_{i,s} \mathbf{y}_{i,p} = \mathbf{y}_{i,p} \mathbf{y}_{i,s}, \\
 \mathbf{y}_{i,s} \sigma_k = \sigma_k \mathbf{y}_{i,s},\ \ s \notin \{ k, k+1 \}\\ \sigma_k \mathbf{y}_{i,k} \sigma_k = \mathbf{y}_{i,{k+1}}, 
 \\ \mathbf{y}_{i,1} \cdots \mathbf{y}_{i,{\ell}} \mathbf{y}_{{i'},1} = \mathbf{y}_{{i'},1} \mathbf{y}_{i,1} \cdots \mathbf{y}_{i,{\ell}},\ \  \ \ \mathbf{y}_{{i'},1} \mathbf{y}_{i,2}= \mathbf{y}_{i,2} \mathbf{y}_{{i'},1}
\end{gather*}
for $k \in [1, \ell-2 ]$ in the second relation; $i, i' \in [1, m]$; $k,j \in [1, \ell-1 ]$ otherwise; and $p,s \in [1, \ell ]$.\hfill
\end{defn}

Notice, when $m=2$, the definition above is precisely Definition \ref{classicaluntwistedtoroidalhecke}.  Therefore, by a similar argument used to prove Lemma \ref{torheckeisom}, we have $\on{Aff}^m(\Sl) \cong \mathbb{C}\left[\left(\Z^{\ell} \right)^m \rtimes \Sl\right]$.

The group $\on{Aff}^m(\Sl)$ contains subgroups $\asl = \on{Aff}_i(\Sl)\cong \mathbb{C}[\Z^{\ell}\rtimes S_{\ell}]$ 
, with $\Z^{\ell}$ appearing in the $i^{th}$ component only (and other components being zero). Similarly, use $ \nu_i: \on{Aff}_i(\Sl)\to \on{Aff}^m(\Sl)$ to be the embedding and let $ \nu_0: S_\ell \to \on{Aff}_i(\Sl)$ and $ \nu^m_0: \Sl\to \on{Aff}^m(\Sl)$ be the embeddings. 
Then there exist evaluation homomorphisms $ev_i: \on{Aff}_i(\Sl)\to \Sl$ sending the $\mathbf{y}_p \to 1$ for $p \in [1, \ell ]$. We have $\nu^m=\nu_i\circ \nu_0$ and 
\[ [\nu_i\ker(ev_i), \nu_j\ker(ev_j)]=1\]
in the group $ \on{Aff}^m(\Sl)$. Then we can prove a characterization of the group $\on{Aff}^m(\Sl)$ analogous to Proposition \ref{liftingprop}. 

\begin{prop} Let $ G$ be any group. For any system of group homomormphisms: $\rho_i: \on{Aff}_i(\Sl) \to G$ satisfying the conditions 
\begin{itemize}
\item[(i)] $\rho_i\circ \nu_0=\rho_j\circ\nu_0$, (so denote $ \rho_0=\rho_i\circ \nu_0: \Sl\to G$), and 
\item[(ii)] $ \rho_i\ker(ev_i) \rho_j\ker(ev_j)= \rho_j\ker(ev_j) \rho_i\ker(ev_i) $ for all $ i, j$, 
\end{itemize}
there is a unique group homomorphism $\rho: \on{Aff}^m(\Sl)\to G$ such that $\rho_i=\rho\circ \nu_i$. 
\end{prop}

\begin{proof}
    We will prove this by induction on $m$. The statement is obvious when $m=1$. Assuming the  statement for $m=k-1$ we prove the result for case $m=k$. Let $\sigma \in \Sl$ and $\bor\in \left(\Z^{\ell} \right)^k$ with $\bor = (r_1, \ldots , r_k)$. Using the inductive assumption, define $$\rho (\bor, \sigma) := \rho \big((r_1, \ldots , r_{k-1}),1 \big) \rho_k (r_k, \sigma).$$ The fact that $\rho$ is well-defined, unique, and satisfies the condition $\rho_i=\rho\circ \nu_i$  follow from the definition once we prove that $\rho$ is shown to be a homomorphism of groups. This is done as follows: let $\tau\in\Sl$ and $\mathbf q=(q_1,\ldots, q_k)$. Then
	\begin{align*}
        \rho((\bor,\sigma)(\boq,\tau))
        &= \rho(\bor+\sigma(\boq),\sigma \tau) \\
        &= \rho \big((r_1+\sigma(q_1), \ldots , r_{k-1}+\sigma(q_{k-1})),1 \big) \rho_k (r_k+\sigma(q_k), \sigma \tau) \\
        &= \rho_1(r_1+\sigma(q_1),1) \cdots \rho_{k-1}(r_{k-1}+\sigma(q_{k-1}),1) \rho_k (r_k+\sigma(q_k), \sigma \tau) \\
        &= \rho_1(r_1,1) \rho_1(\sigma(q_1),1) \cdots \rho_{k-1}(r_{k-1},1)\rho_{k-1}(\sigma(q_{k-1}),1) \rho_k (r_k,1) \rho_k(\sigma(q_k), \sigma \tau) \\
        &= \rho_1(r_1,1) \cdots \rho_k(r_k,1) \rho_1(\sigma(q_1),1) \cdots \rho_{k-1}(\sigma(q_{k-1}),1) \rho_k(\sigma(q_k), \sigma \tau) \hspace{20mm} (*)
    \end{align*}
    where the third equality follows from the inductive definition of $\rho$, and the last equality from condition (ii). Now note that, for any $j \in [1,k]$, by condition (i) we have $\rho_0(\sigma) \rho_j(q_j,1) = \rho_j(\sigma(q_j),\sigma)=\rho_j(\sigma(q_j),1) \rho_0 (\sigma)$. Thus $(*)$ is
	\begin{align*}
        \rho_1(r_1,1) & \cdots \rho_k(r_k,1) \rho_1(\sigma(q_1),1) \cdots \rho_{k-1}(\sigma(q_{k-1}),1) \rho_0(\sigma) \rho_k(q_k, \tau) \\
        &= \rho_1(r_1,1) \cdots \rho_k(r_k,1) \rho_1(\sigma(q_1),1) \cdots \rho_0(\sigma) \rho_{k-1}(q_{k-1},1) \rho_k(q_k, \tau) \\
        &= \rho_1(r_1,1) \cdots \rho_k(r_k,1) \rho_0(\sigma) \rho_1(q_1,1) \cdots \rho_{k-1}(q_{k-1},1) \rho_k(q_k, \tau) \\
        &= \rho_1(r_1,1) \cdots \rho_{k-1}(r_{k-1},1) \rho_k(r_k,\sigma) \rho_1(q_1,1) \cdots \rho_{k-1}(q_{k-1},1) \rho_k(q_k, \tau) \\
        &= \rho \big((r_1, \ldots , r_{k-1}),1 \big) \rho_k (r_k, \sigma) \rho \big((q_1, \ldots , q_{k-1}),1 \big) \rho_k (q_k, \tau) \\
        &= \rho(\bor,\sigma) \rho(\boq,\tau).
    \end{align*}
\end{proof}

\begin{cor} Let $ M$ be an $\Sl$--module over $ \cc$. Assume that there is an $\on{Aff}_i(\Sl)$--module structure on $M$, denoted by $M_i$, extending the $\Sl$--module structure on $M$. Then there is a unique $\on{Aff}^m(\Sl)$--module structure on $M$ whose restriction to $\on{Aff}_i(\Sl)$ via $ \nu_i$ is $M_i$. \label{multiheckeliftingcorr}
\end{cor}

\begin{proof}
Define an $\on{Aff}^m(\Sl)$--module structure on $M$ by declaring that $(n_1, n_2, \ldots , n_m, \sigma)$ act as $(n_1, 1) (n_2, 1) \cdots (n_m, 1) (0, \sigma)$ with each $(n_i,1)$ denoting the $M_i$ structure and $(0 , \sigma)$ the $\Sl$--module structure. Since $(0, \sigma)(n,1) = (\sigma(n),\sigma) = (\sigma(n),1)(0,\sigma)$, we have that the product $(n_1, \ldots , n_m, \sigma) (n'_1, \ldots , n'_m, \sigma') = (n_1+\sigma(n'_1), \ldots , n_m+\sigma(n'_m), \sigma \sigma')$, and thus this action does indeed define an $\on{Aff}^m(\Sl)$--module structure. That its restriction to $\on{Aff}_i(\Sl)$ via $ \nu_i$ is $M_i$ is obvious. Uniqueness is a consequence of Lemma \ref{uniqueliftedmodstruc} (b) since the $\on{Aff}_i(\Sl)$ generate $\on{Aff}^m(\Sl)$ as an associative algebra.
\end{proof}

\subsection{The Schur-Weyl duality for multiloop Lie algebras of type $A$} We prove the  main theorem of this section. It shows we can extend a known Schur-Weyl duality in the single loop case to the multiloop case by ``gluing together several loops'' at once, and it gives an alternative but less explicit way to  extend Schur-Weyl duality to the toroidal setting.  However  this method  is less reliant on the presentation of toroidal algebras as the affinization of affine algebras.

\begin{thm}\label{mainliftingthm}
Let  $\ell, m\geq 1$ be integers with $\g=\sln$ and $ \ell\leq n$. Then
\begin{itemize}
\item[(a)] For any $ \cc[\on{Aff}^m(\Sl)] $--module $M$
there is a unique $L^m(\frak g)$--module structure on $ \mathcal F(M)=M\otimes_{\cc[\Sl]} V^{\otimes {\ell}}$ extending the $L_i(\frak g)$--module structures given by Theorem \ref{nonquantumuntwistedequiv}.
\item[(b)] $\mathcal F:  \on{Mod-}\cc[\on{Aff}^m(\Sl)]\to 
L^m(\g)\on{-Mod}$ is a fully faithful functor.
\item[(c)] Let $L^m(\frak g)\on{-Mod}^{\ell}$ be the full subcategory of  all $L^m(\frak g)$--modules whose restriction to $ \frak g$ via $\iota_0^m$ is locally finite dimensional and all composition factors are the composition factors of $ V^{\otimes {\ell}}$. Then $ \mathcal F: \on{Mod-}\cc[\on{Aff}^m(\Sl)]\to 
L^m(\frak g)\on{-Mod}^{\ell}$ is a category equivalence. 
\end{itemize}
\end{thm}

\begin{proof}
\begin{itemize}
\item[(a)] By Corollary \ref{multiheckeliftingcorr} we see that the restricted $\on{Aff}_i(S_{\ell})$--module structures $M_i, 1\le i \le m$ combine to give a unique $\on{Aff}^m(S_\ell)$--module structure, which must therefore be $M$.  Let $L=\mathfrak{gl}\left(\mathcal{F} (M) \right)$, and define $\rho_i:L_i(\frak g)\to L, 1\le i \le m$ by $\rho_i(x \otimes t_i^{n_i}) (m \otimes \mathbf{v}) = (x \otimes t_i^{n_i}).(m \otimes \mathbf{v})$ and check that the $\rho_i$ satisfy the conditions of Proposition \ref{liftingprop}.
\begin{itemize}
\item[(i)] For $x \in \g$, $\rho_i\circ \iota_{i,0} (x) = \rho_i (x \otimes t_i^0) = \rho_j (x \otimes t_j^0) = \rho_j\circ \iota_{j,0} (x)$.
\item[(ii)] We show this in two steps. First, we show that for any sequence $ (i_1, \dots, i_c)$ of elements of $\{1, \ldots, m\}$,
 and for any $x_1, \ldots, x_c\in \frak g$, we have
$$\big([\rho_{i_1}(x_1\otimes t_{i_1}^{p_1}), [\rho_{i_2}(x_2\otimes t_{i_2}^{p_2}), \cdots[\rho_{i_{c-1}}(x_{c-1}\otimes t_{i_{c-1}}^{p_{c-1}}), \rho_{i_c}(x_c\otimes t_{i_c}^{p_c})]\cdots ]]\big).(m \otimes \mathbf{v})$$
$$= \sum _{j=1}^\ell\left (m . \mathbf{y}_{i_1,j}^{p_1}\mathbf{y}_{i_2,j}^{p_2} \cdots \mathbf{y}_{i_c,j}^{p_c} \otimes [x_1, [x_2, \cdots[x_{c-1}, x_c]\cdots ]]_j . \mathbf{v}\right )$$
where $\mathbf y_{i,j}^p, 1 \le i \le m, 1\le j \le \ell, p \in \mathbb Z$ is shorthand for $(p\mathbf e_{ij},1)\in \text{Aff}^m(S_\ell)$ where $\mathbf e_{ij}$ is the $m$-tuple with the unit vector $\mathbf e_j$ in the $i$th position and 0's elsewhere. We use induction on $c$ with the $c=1$ case following from Theorem \ref{nonquantumuntwistedequiv}. We compute: 
\begin{align*}
&\big([\rho_{i_1}(x_1\otimes t_{i_1}^{p_1}), [\rho_{i_2}(x_2\otimes t_{i_2}^{p_2}), \cdots[\rho_{i_{c-1}}(x_{c-1}\otimes t_{i_{c-1}}^{p_{c-1}}), \rho_{i_c}(x_c\otimes t_{i_c}^{p_c})]\cdots ]]\big).(m \otimes \mathbf{v})\\
&=\sum _{j=1}^\ell\rho_{i_1}(x_1\otimes t_{i_1}^{p_1})\left (m. \mathbf{y}_{i_2,j}^{p_2}\mathbf{y}_{i_3,j}^{p_3} \cdots \mathbf{y}_{i_c,j}^{p_c} \otimes [x_2, [x_3, \cdots[x_{c-1}, x_c]\cdots ]]_j . \mathbf{v}\right )\\
&\qquad-\sum _{k=1}^\ell[\rho_{i_2}(x_2\otimes t_{i_2}^{p_2}), \cdots[\rho_{i_{c-1}}(x_{c-1}\otimes t_{i_{c-1}}^{p_{c-1}}), \rho_{i_c}(x_c\otimes t_{i_c}^{p_c})]\cdots ]\left (m. \mathbf{y}_{i_1,k}^{p_1} \otimes (x_{1})_k. \mathbf{v}\right )\\
&=\sum_{j=1}^\ell\sum_{k=1}^\ell \big (m. \mathbf{y}_{i_2,j}^{p_2}\mathbf{y}_{i_3,j}^{p_3} \cdots \mathbf{y}_{i_c,j}^{p_c}\mathbf{y}_{i_1,k}^{p_1} \otimes (x_1)_k[x_2, [x_3, \cdots[x_{c-1}, x_c]\cdots ]]_j. \mathbf{v}\\
&\qquad-m. \mathbf{y}_{i_1,k}^{p_1}\mathbf{y}_{i_2,j}^{p_2}\mathbf{y}_{i_3,j}^{p_3} \cdots \mathbf{y}_{i_c,j}^{p_c} \otimes [x_2, \cdots[x_{c-1}, x_c]\cdots ]_j(x_{1})_k. \mathbf{v}\big ).
\end{align*}
The last line is 0 unless $j=k$, in which case it becomes:
\begin{align*}
\sum_{j=1}^\ell\big (m. \mathbf{y}_{i_1,j}^{p_1}\mathbf{y}_{i_2,j}^{p_2}\mathbf{y}_{i_3,j}^{p_3} \cdots \mathbf{y}_{i_c,j}^{p_c} \otimes [x_1,[x_2, [x_3, \cdots[x_{c-1}, x_c]\cdots ]]]_j. \mathbf{v}\big ).
\end{align*}

It follows that 
$$
\sum_{x_1,\dots, x_c \in \mathfrak g} [\rho_{i_1}(x_1\otimes t_{i_1}^{p_1}), [\rho_{i_2}(x_2\otimes t_{i_2}^{p_2}), \cdots[\rho_{i_{c-1}}(x_{c-1}\otimes t_{i_{c-1}}^{p_{c-1}}), \rho_{i_c}(x_c\otimes t_{i_c}^{p_c})]\cdots ]]=0$$
if
$$
\sum_{x_1,\dots, x_c\in \mathfrak{g}} [x_1, [x_2, \cdots[x_{c-1}, x_c]\cdots ]]=0.
$$
\item[(iii)]
As in the previous step, we see that
\begin{align*}
&\big([\rho_{i_1}(x_1\otimes t_{i_1}^{p_{i_1}}), [\rho_{i_2}(x_2\otimes t_{i_2}^{p_{i_2}}), \cdots[\rho_{i_{c-1}}(x_{c-1}\otimes t_{i_{c-1}}^{p_{i_{c-1}}}), \rho_{i_c}(x_c\otimes t_{i_c}^{p_{i_c}})]\cdots ]]\big).(m \otimes \mathbf{v})\\
&\qquad=\sum_{k=1}^\ell\big (m. \mathbf{y}_{i_1,k}^{p_{i_1}}\mathbf{y}_{i_2,k}^{p_{i_2}}\mathbf{y}_{i_3,k}^{p_{i_3}} \cdots \mathbf{y}_{i_c,k}^{p_{i_c}} \otimes [x_1,[x_2, [x_3, \cdots[x_{c-1}, x_c]\cdots ]]]_k. \mathbf{v}\big ).
\end{align*}
If $(j_1,\dots, j_c)$ is a permutation of $(i_1,\dots i_c)$ then 
\begin{align*}
&\big([\rho_{j_1}(x_1\otimes t_{j_1}^{p_{j_1}}), [\rho_{j_2}(x_2\otimes t_{j_2}^{p_{j_2}}), \cdots[\rho_{j_{c-1}}(x_{c-1}\otimes t_{j_{c-1}}^{p_{j_{c-1}}}), \rho_{j_c}(x_c\otimes t_{j_c}^{p_{j_c}})]\cdots ]]\big).(m \otimes \mathbf{v})\\
&\qquad=\sum_{k=1}^\ell\big (m. \mathbf{y}_{j_1,k}^{p_{j_1}}\mathbf{y}_{j_2,k}^{p_{j_2}}\mathbf{y}_{j_3,k}^{p_{j_3}} \cdots \mathbf{y}_{j_c,k}^{p_{j_c}} \otimes [x_1,[x_2, [x_3, \cdots[x_{c-1}, x_c]\cdots ]]]_k. \mathbf{v}\big )
\end{align*}
which is the same element of $M\otimes V^{\otimes \ell}$ since the $\mathbf y_{i,j}$ commute.  
\end{itemize}
Therefore, the $\rho_i$ thus defined satisfy the conditions of Proposition \ref{liftingprop}. The existence of the $L^m(\frak g)$--module structure follows from Proposition \ref{liftingprop}, and uniqueness follows from Lemma \ref{uniqueliftedmodstruc}.

\item[(b)] As the functor is fully faithful in the $m=1$ case, this statement follows from Theorem \ref{nonquantumuntwistedequiv}, Proposition \ref{liftingprop}, and Lemma \ref{uniqueliftedmodstruc}.

\item[(c)] We have the restrictions $\mathcal{F}(M) \to \left( \mathcal{F}(M)|_{L_i (\g)} \right) \big|_{\g}$ via $\iota_0^m$ which each satisfy the stated conditions. Then the statement follows from Theorem \ref{nonquantumuntwistedequiv}, Proposition \ref{liftingprop}, and Lemma \ref{uniqueliftedmodstruc}.

\end{itemize}
\end{proof}

\end{document}